\documentclass[12pt,runningheads]{amsart}
\usepackage{mathrsfs, amssymb, upgreek}
\usepackage{longtable}

\usepackage{amsthm}
\usepackage{multirow}
\usepackage{graphics}
\usepackage{array}
\usepackage{setspace}
\usepackage{tabularx}
\usepackage[all, matrix, arrow, curve]{xy}

\newcommand{\CC}{\mathbb{C}}
\newcommand{\QQ}{\mathbb{Q}}

\newcommand{\ZZ}{\mathbb{Z}}
\newcommand{\PP}{\mathbb{P}}

\newcommand{\ff}{\mathfrak{f}}
\newcommand{\XX}{\mathfrak{X}}
\newcommand{\XD}{\mathfrak{D}}
\newcommand{\OOO}{{\mathscr{O}}} 
\newcommand{\mumu}{{\boldsymbol{\mu}}}

\newcommand{\Aut}{\operatorname{Aut}}

\newcommand{\Spec}{\operatorname{Spec}}
\newcommand{\Tot}{\operatorname{Tot}}
\newcommand{\LLL}{\mathscr{L}}

\newcommand{\rT}{\operatorname{T}}
\newcommand{\rA}{\operatorname{A}}

\newcommand{\mult}{\operatorname{mult}}
\newcommand{\Eu}{\operatorname{Eu}}

\newcommand{\Diff}{\operatorname{Diff}}
\newcommand{\Sing}{\operatorname{Sing}}
\newcommand{\Supp}{\operatorname{Supp}}
\newcommand{\Pic}{\operatorname{Pic}}
\newcommand{\bb}{\operatorname{b}}

\newcommand{\lcm}{\operatorname{lcm}}
\newcommand{\K}{\operatorname{K}}
\renewcommand{\emptyset}{\varnothing}

\newcommand{\xref}[1]{\textup{\ref{#1}}}
\newcommand{\type}[1]{$\mathrm{#1}$}

\makeatletter
\@addtoreset{equation}{subsection}
\makeatother

\swapnumbers
\theoremstyle{plain}
\newtheorem{theorem}[subsection]{Theorem}
\newtheorem{lemma}[subsection]{Lemma}
\newtheorem{proposition}[subsection]{Proposition}

\newtheorem{scorollary}[equation]{Corollary}
\newtheorem*{claim*}{Claim}

\newtheorem{claim}[subsection]{Claim}
\newtheorem{slemma}[equation]{Lemma}
\newtheorem{sproposition}[equation]{Proposition}
\newtheorem{stheorem}[equation]{Theorem}

\theoremstyle{definition}
\newtheorem{definition}[subsection]{Definition}
\newtheorem{case}[subsection]{}
\newtheorem{notation}[subsection]{Notation}
\newtheorem{assumption}[subsection]{Assumption}

\newtheorem{sremark}[equation]{Remark}

\newtheorem{construction}[subsection]{Construction}
\newtheorem{example}[subsection]{Example}
\newtheorem{sexample}[equation]{Example}
\newtheorem{scase}[equation]{}

\newcounter{NN}
\renewcommand{\theNN}{{\rm\arabic{NN}${}^o_{}$}}
\def\nr{\refstepcounter{NN}{\theNN}}%

\newcounter{NNN}
\renewcommand{\theNNN}{{\rm\arabic{NNN}${}^{\#}_{}$}}
\def\nrr{\refstepcounter{NNN}{\theNNN}}%

\newcounter{NNNN}
\renewcommand{\theNNNN}{{\rm\arabic{NNNN}${}^{\bullet}_{}$}}
\def\nrrr{\refstepcounter{NNNN}{\theNNNN}}%

\renewcommand\labelenumi{{\rm (\roman{enumi})}}
\renewcommand\theenumi{{\rm (\roman{enumi})}}

\author{Yuri Prokhorov}
\title[Log canonical degenerations]{Log canonical degenerations of del Pezzo surfaces in $\QQ$-Gorenstein families}

\address{\noindent
Steklov Mathematical Institute, Russia
\newline\indent
Moscow State Lomonosov 
University, Russia
\newline\indent
National Research University Higher School of Economics, Russia
}
\email{prokhoro@mi.ras.ru}
 \thanks{The author was partially supported by 
the RFBR grants 15-01-02164, 15-01-02158,
by the Russian Academic Excellence Project '5-100', and by
RSF grant, project 14-21-00053 dated 11.08.14.}


\begin{document}
\maketitle
\begin{abstract}
We classify del Pezzo surfaces of Picard number one with
log canonical singularities admitting $\QQ$-Gorenstein smoothings.
\end{abstract}

\section{Introduction}
Throughout this paper we work over the complex number field $\CC$.
A \textit{smoothing} of a surface $X$ is a flat family $\XX\to\XD$
over a unit disk $0\in\XD\subset\CC$ such that the fiber $\XX_0$
is isomorphic to $X$ and the general fiber is smooth.
In this situation $X$ can be considered as a degeneration of a fiber $\XX_t$, $0\neq t\in\XD$.
A smoothing is said to be $\QQ$-\textit{Gorenstein} if so the total family $\XX$ is.
Throughout this paper a \textit{del Pezzo surface} means a normal projective
surface whose anticanonical divisor is $\QQ$-Cartier and ample.
We study $\QQ$-Gorenstein smoothings of del Pezzo surfaces with log canonical singularities.
This is interesting for applications to birational geometry and the minimal model program 
(see e.g. \cite{Mori-Prokhorov-2008d}, \cite{Prokhorov-e-QFano7}) as well as to moduli problems \cite{Kollar-ShB-1988}, \cite{Hacking2004}.
Smoothings of del Pezzo surfaces with log terminal singularities were considered in \cite{Manetti-1991},
\cite{Hacking-Prokhorov-2010}, \cite{Prokhorov-degenerations-del-Pezzo}.

\begin{theorem}
\label{theorem-main} 
Let $X$ be a del Pezzo surface with only log canonical singularities and $\uprho(X)=1$. 
Assume that $X$ admits a $\QQ$-Gorenstein smoothing and there exists at least one non-log terminal point $(o\in X)$. 
Let $\eta : Y\to X$ be the minimal resolution. Then there is a rational curve fibration 
$\varphi : Y\to T$ over a smooth curve $T$ such that a component $C_1$ of the $\eta$-exceptional divisor 
dominating $T$
is unique, it is a section of $\varphi$, and its discrepancy equals $-1$. 
Moreover, $o$ is the only non-log terminal singularity and singularities of $X$ 
outside $o$ are at worst Du Val of type \type {A}. The surface
$X$ and singular fibers of $\varphi$ are described in the table below. 

All the cases except possibly for 
\ref{types-I=2-2I2-IV} with $5\le n\le 8$ and~\ref{types-I=2-2IV} with $5\le n\le 10$ occur.
\end{theorem}
\par\medskip\noindent
\scalebox{0.9}{
\setlength{\extrarowheight}{6pt}
\begin{tabularx}{\textwidth}{c|c|c|c|c|c|c}
&\multicolumn2{c|}{singularities }&\multirow2{*}{$\uprho(Y)$ }&
\multirow2{*}{$K_X^2$ }&
\multirow2{*}{singular}&
\multirow2{*}
{condition}
\\
\cline{2-3}
& $(o\in X)$&$X\setminus\{o\}$&&&fibers of $\varphi$&on $n$
\\
\hline
\nr{}
\label{types-simple-elliptic}
& $\mathrm{Ell}_n$& $\emptyset$& $2$& $n$& $\emptyset$ & $n\le 9$
\\
\nr{}
\label{types-I=2-4-I2}
& $[n; [2]^4]$& $4\rA_1$& $10$& $n-2$& $4\mathrm{(I_2)}$ & $3\le n\le6$
\\
\nr{}
\label{types-I=2-2I2-IV}
&$[n,2,2; [2]^4]$& $2\rA_1$& $10$& $n-2$& $2\mathrm{(I_2)}\mathrm{(II)}$ & $3\le n\le 8$
\\
\nr{}
\label{types-I=2-2IV}
& $[2,2,n,2,2; [2]^4]$&$\emptyset$& $10$& $n-2$& $2\mathrm{(II)}$ & $3\le n\le10 $
\\
\nr{}
\label{types-I=3}&
$[n; [3]^3]$&$3\rA_2$& $11$& $n-1$& $3\mathrm{(I_3)}$ & $2,3,4$
\\
\nr{}
\label{types-I=4-I22I4}& 
$[n; [2],[4]^2]$ & $\rA_1$, $2\rA_3$& $12$& $n-1$ & $\mathrm{(I_2)}2\mathrm{(I_4)}$ & $2,3$
\\
\nr{}
\label{types-I=6}& 
$[2; [2],[3],[6]]$&$\rA_1$, $\rA_2$, $\rA_5$& $13$& $1$& $\mathrm{(I_2)}\mathrm{(I_3)}\mathrm{(I_6)}$ &
\end{tabularx}
}
\par\medskip\noindent
For a precise description the surfaces that occur in our classification we refer to Sect. 
\ref{section-fibrations}.

To show the existence of $\QQ$-Gorenstein smoothings we use unobstructedness 
of deformations (see Proposition~\ref{no-obstructions}) and 
local investigation of $\QQ$-Gorenstein smoothability 
of log canonical singularities:
\begin{theorem}
\label{Theorem-Q-smoothings}
Let $(X\ni P)$ be a strictly log canonical surface singularity of index $I>1$ admitting a 
$\QQ$-Gorenstein smoothing. Then it belongs to one of the following types:
\par\medskip\noindent\setlength{\extrarowheight}{1pt}
\begin{tabularx}{\textwidth}{l|l|l|l|l|l}
\hline
&$I$ & $(X\ni P)$&{\rm condition}& $\upmu_P$&$-\K^2$
\\\hline
\nrrr
\label{smoothing-2}&$2$& $[n_1,\dots, n_s; [2]^4]$& $\sum (n_i-3)\le 3$& $4-\sum (n_i-3)$& $\sum (n_i-2)$\\
\nrrr
\label{smoothing-3}&$3$ &$[n; [3],[3],[3]]$& $n=2,3,4$& $4-n$&$n$\\
\nrrr
\label{smoothing-4}&$4$ & $[n; [2],[4],[4]]$& $n=2,3$& $3-n$&$n+1$\\
\nrrr
\label{smoothing-6}&$6$ &$[2; [2],[3],[6]]$&&$0$&$4$
\end{tabularx}
where $\upmu_P$ is the Milnor fiber of the smoothing.

$\QQ$-Gorenstein smoothings exist in cases~\ref{smoothing-3},
\ref{smoothing-4},~\ref{smoothing-6}, as well as in the case~\ref{smoothing-2}
for singularities of types $[n; [2]^4]$ with $n\le 6$,
$[n_1,\dots, n_s; [2]^4]$ with $\sum (n_i-2)\le 2$, $[4,3; [2]^4]$, and $[3,3,3; [2]^4]$.
In all other cases the existence of $\QQ$-Gorenstein smoothings is unknown.
\end{theorem}

Smoothability of log canonical singularities of index $1$ were studied earlier 
(see e.g. \cite[Ex. 6.4]{Looijenga-Wahl-1986}, \cite[Corollary 5.12]{Wahl1980}.

As a bi-product we construct essentially canonical threefold 
singularities of index $5$ and $6$.
We say that a canonical singularity $(\XX\ni o)$ is \textit{essentially canonical} 
if there exist a crepant divisor with center $o$. 
V.~Shokurov conjectured that essentially canonical singularities of 
given dimension have bounded indices. 
This is well-known in dimension two: canonical surface singularities are
Du Val and their index equals $1$.
Shokurov's conjecture was proved 
in dimension three by M. Kawakita \cite{Kawakita-index}.
More precisely, he proved that the index of an essentially canonical
threefold singularity is at most $6$.
The following theorem supplements Kawakita's result.
\begin{theorem}
\label{theorem-main-index}
For any $1\le I\le 6$ there exist 
a three-dimensional essentially
canonical singularity of index $I$.
\end{theorem}

In fact our result is new only for $I=5$ and $6$: \cite{Hayakawa-Takeuchi-1987} 
classified threefold canonical hyperquotient singularities
and among them there are examples satisfying conditions of 
our theorem with $I \le 4$.
Theorem~\ref{theorem-main-index} together with \cite{Kawakita-index} gives the following

\begin{theorem}
Let $\mathfrak I$ be the set of indices of three-dimensional essentially
canonical singularities. Then 
\begin{equation*}
\mathfrak I=\{1,2,3,4,5,6\}.
\end{equation*}
\end{theorem}

The paper is organized as follows.
Sect.~\ref{sect-lc} is preliminary. In Sect.~\ref{sect-smoothings}
we obtain necessary  conditions for $\QQ$-Gorenstein smoothability
of two-dimensional  log canonical singularities. 
In Sect.~\ref{section-Examples} we construct examples of $\QQ$-Gorenstein smoothings.
Theorem~\ref{theorem-main-index} will be proved in Sect.
\ref{sect-Indices}.
In Sect.~\ref{sect-Noether} 
we collect important results on del Pezzo surfaces 
admitting $\QQ$-Gorenstein smoothings. The main birational construction 
for the proof of Theorem~\ref{theorem-main}
is outlined in Sect.~\ref{section-Del-Pezzo-surfaces}. 
which  will be considered in Sect.
\ref{section-fibrations} and
\ref{section-del-pezzo}.

\par\medskip\noindent
\textbf{Acknowledgments. } 
I thank Brendan Hassett whose questions encouraged me to write up
my computations. The questions were asked during Simons Symposia ``Geometry Over Nonclosed Fields, 2016''.
I am grateful to the organizers of this activity for the invitation and  creative atmosphere.
I also would like to thank the referee for careful reading and numerous helpful comments and suggestions.

\section{Log canonical singularities}\label{sect-lc}
For basic definitions and terminology of the minimal model program, 
we refer to \cite{Kollar-Mori-1988} or \cite{Utah}.
\begin{case}
\label{notation-singularities}
Let $(X\ni o)$ be a log canonical surface singularity.
The \textit{index} of $(X\ni o)$ is the smallest positive integer $I$
such that $IK_X$ is Cartier. 
We say that $(X\ni o)$ is \textit{strictly log canonical}
if it is log canonical but not log terminal.
\end{case}

\begin{definition}
A normal Gorenstein surface singularity is said to be \textit{simple
elliptic} if the exceptional divisor of the minimal resolution is a smooth elliptic
curve. We say that a simple
elliptic singularity is of type $\mathrm{Ell}_n$ if the self-intersection of 
the exceptional divisor equals $-n$.

A normal Gorenstein surface singularity is called a \textit{cusp} if the
exceptional divisor of the minimal resolution is a cycle of smooth rational curves
or a rational nodal curve.
\end{definition}

\begin{case}
We recall a notation on weighted graphs.
Let $(X\ni o)$ be a rational surface singularity, let
$\eta:Y\to X$ be its minimal resolution, and let
$E=\sum E_i$ be the exceptional divisor.
Let $\Gamma=\Gamma(X,o)$ be the dual graph of $(X\ni o)$,
that is,  $\Gamma$ is a weighted graph whose 
vertices correspond to exceptional prime divisors $E_i$
and edges join vertices meeting each other.
In the usual way we attach to each vertex $E_i$ the number $-E_i^2$.
Typically, we omit $2$ if $-E_i^2=2$.

If $(X\ni o)$ is a cyclic quotient singularity of type $\frac 1r (1,q)$, $\gcd(r,q)=1$,
then the graph $\Gamma$ is a chain:
\begin{equation}
\label{equation-chain}
\xy
\xymatrix@C=38pt{
&\underset{n_1}\circ\ar@{-}[r]&\underset{n_2}\circ\ar@{-}[r]
&\cdots\ar@{-}[r]&\underset{n_k}\circ
}
\endxy
\end{equation}
We denote it by $[n_1,\dots,n_k]=\langle r,\, q\rangle$. The numbers $n_i$ are determined by the expression 
of $r/q$ as a continued fraction \cite{Brieskorn-1967-1968}. 
For positive integers $n$, $r_i$, $q_i$,\, $\gcd(r_i,q_i)=1$, $i=1,\dots,s$, the symbol
\begin{equation*}
\langle n; r_1,\dots, r_s;\, q_1,\dots, q_s\rangle
\end{equation*}
denotes the following graph
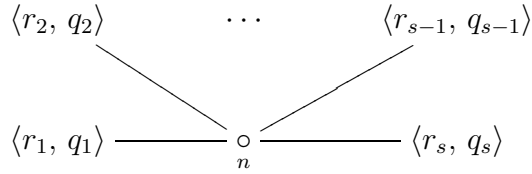
\begin{equation*}
\xy
\xymatrix@C=38pt{
\langle r_2,\, q_2\rangle&\cdots&\langle r_{s-1},\, q_{s-1}\rangle
\\
\langle r_1,\, q_1\rangle&\underset{n}\circ\ar@{-}[r]\ar@{-}[ul]\ar@{-}[ur]\ar@{-}[l]&\langle r_s,\, q_s\rangle
}
\endxy
\end{equation*}
For short, we will omit $q_i$'s:$\langle n; r_1,\dots, r_s\rangle$.
If $\langle r_i,\, q_i\rangle=[n_{i,1},n_{i,2},\dots]$, then we also denote
\begin{equation*}
\langle n; r_1,\dots, r_s;\, q_1,\dots, q_s\rangle=
[n; [n_{1,1},n_{1,2},\dots],\dots, [n_{s,1},n_{s,2},\dots]].
\end{equation*}
For example,
$\langle n; 3,3,3; 1,1, 2\rangle=[n; [3], [3], [2,2]]$ is the graph:
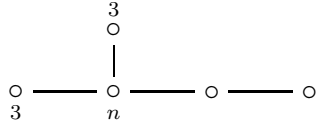
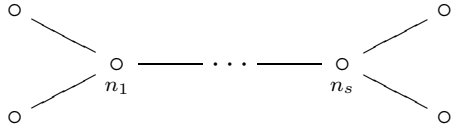
\begin{equation*}
\xy
\xymatrix@R=3pt{
&\overset3\circ&
\\
\underset3\circ&\underset{n}\circ\ar@{-}[r]\ar@{-}[u]\ar@{-}[l]&\underset{}\circ\ar@{-}[r]&\underset{}\circ
}
\endxy
\end{equation*}
The graph
\begin{equation*}
\xy
\xymatrix@R=3pt{
\circ&&&&\circ
\\
&\underset{n_1}\circ\ar@{-}[r]\ar@{-}[lu]\ar@{-}[ld]&\cdots\ar@{-}[r]&\underset{n_s}\circ\ar@{-}[ru]\ar@{-}[rd]
\\
\circ&&&&\circ
}
\endxy
\end{equation*}
will be denoted by $[n_1,\dots,n_s; [2]^4]$. 
\end{case}

\begin{theorem}[{\cite[\S 9]{Kawamata-1988-crep}}]
\label{theorem-classification-lc-singularities}
Let $(X\ni o)$ be a strictly log canonical surface singularity of index $I$.
Then one of the following holds:
\begin{enumerate}
\item
\label{Theorem-simple-elliptic-cusp=I=1}
$I=1$ if and only if $(X\ni o)$ is either a simple elliptic singularity or a 
cusp,
\item 
$I=2$ if and only if $\Gamma(X,o)$ is of type $[n_1,\dots,n_s; [2]^4]$, $s\ge 1$,
\item 
$I=3$ if and only if $\Gamma(X,o)$ is of type $\langle n; 3,3,3\rangle$,

\item 
$I=4$ if and only if $\Gamma(X,o)$ is of type $\langle n; 2,4,4\rangle$,

\item 
$I=6$ if and only if $\Gamma(X,o)$ is of type $\langle n; 2,3,6\rangle$.
\end{enumerate}
\end{theorem}

\begin{scorollary}
A strictly log canonical surface singularity 
is not rational if and only if it is of index $1$.
\end{scorollary}

\begin{case}
Let $(X\ni o)$ be a strictly log canonical surface singularity of index $I$, 
let 
$\eta: Y\to X$ be its minimal resolution, and let $E=\sum E_i$
be the exceptional divisor. Let us contract all the components of $E$ 
with discrepancies $>-1$:
\begin{equation}
\label{equation-min-resolution}
\eta: Y \overset{\tilde \eta}\longrightarrow \tilde X \overset{\sigma}\longrightarrow X.
\end{equation}
Let $\tilde C=\sum \tilde C_i:= \tilde \eta_* E$ be the $\sigma$-exceptional divisor.
Then the pair $(\tilde X,\tilde C)$ has only divisorial log terminal singularities (dlt) 
and the following relation holds
\begin{equation*}
K_{\tilde X}= \sigma^* K_X-\tilde C.
\end{equation*}
The extraction $\sigma: \tilde X\to X$ is called the \textit{dlt modification} of $(X\ni o)$.
\end{case}

\begin{scorollary}[see {\cite[\S 9]{Kawamata-1988-crep}}, {\cite[\S 3]{Utah}},
{\cite[\S 4.1]{Kollar-Mori-1988}}, {\cite[\S 6.1]{Prokhorov-2001}}]
\label{proposition-classification-lc-singularities}
In the above notation one of the following holds:
\begin{enumerate}
\item
$I=1$, $\tilde X=Y$ is smooth, and $(X\ni o)$ is either a simple elliptic
or a cusp singularity;
\item
$I=2$, $\tilde C=\sum_{i=1}^s \tilde C_i$ is a chain of smooth rational curves
meeting transversely at smooth points of $\tilde X$ so that $\tilde C_i\cdot \tilde C_{i+1}=1$, and 
the singular locus of $\tilde X$ consists of two Du Val points 
of type \type{A_1} lying on $\tilde C_1$ and two Du Val points 
of type \type{A_1} lying on $\tilde C_s$ \textup(the 
case $s=1$ is also possible and then $\tilde C=\tilde C_1$ 
is a smooth rational curve containing four Du Val points of type \type{A_1}\textup);
\item
$I=3$, $4$, or $6$, $\tilde C$ is a smooth rational curve, the pair $(\tilde X,\tilde C)$ 
has only purely log terminal singularities \textup(plt\textup), and the singular locus of $\tilde X$ consists of
three cyclic quotient singularities of types $\frac 1{r_i}(1, q_i)$, $\gcd(r_i, q_i)=1$ with
$\sum 1/r_i=1$. In this case $I=\lcm (r_1,r_2,r_3)$.
\end{enumerate}
\end{scorollary}

\begin{case}\label{index-one-cover}
Let $(X\ni o)$ be a log canonical singularity of index $I$
(of arbitrary dimension).
Recall (see e.g. \cite[Definition 5.19]{Kollar-Mori-1988}) 
that the \textit{index one cover} of $(X\ni o)$ is a 
finite morphism $\pi:X^\sharp\to X$, where 
\begin{equation*}
X^\sharp:=\operatorname{Spec}\left(\bigoplus _{i=0}^{I-1}\OOO_X(-iK_X)\right).
\end{equation*}
Then $X^\sharp$ is irreducible, $o^\sharp=\pi^{-1}(o)$ is one point,  
$\pi$ is \'etale over $X\setminus\Sing(X)$, and 
$K_{X^\sharp}=\pi^*K_X$ is Cartier.

In this situation, $(X^\sharp\ni o^\sharp)$ is a log canonical singularity of index $1$.
Moreover if $(X\ni o)$ is log terminal (resp. canonical, terminal), then 
so the singularity $(X^\sharp\ni o^\sharp)$ is.
\end{case}

\begin{scorollary}
A strictly log canonical surface singularity of index $I>1$ is a quotient 
of a simple elliptic or 
cusp singularity $(X^\sharp\ni o^\sharp)$ by a cyclic group $\mumu_I$
of order $I=2$, $3$, $4$ or $6$ whose action on $X^\sharp\setminus\{o^\sharp\}$ is free. 
\end{scorollary}

\begin{construction}[see {\cite[Proof of Theorem 9.6]{Kawamata-1988-crep}}]
\label{construction-index-one-cove-log-canonical}
Let $(X\ni o)$ be a strictly log canonical surface singularity of index $I>1$,
let $\pi: (X^\sharp\ni o^\sharp)\to (X\ni o)$ be the index one cover,
and let $\tilde \sigma: (\tilde X^\sharp\supset \tilde C^\sharp)\to (X^\sharp\ni o^\sharp)$ 
be the minimal resolution.
The action of $\mumu_I$ lifts to $\tilde X^{\sharp}$ so that 
the induced action on $\OOO_{\tilde X^{\sharp}}(K_{\tilde X^{\sharp}}+ \tilde C^{\sharp})=
\tilde \sigma^*\OOO_{X^\sharp}(K_{X^\sharp})$
and $H^0(\tilde C^{\sharp}, \OOO_{\tilde C^{\sharp}}(K_{\tilde C^{\sharp}}))$
is faithful. Let $(\tilde X\supset \tilde C):= (\tilde X^\sharp\supset \tilde C^\sharp)/\mumu_I$.
Thus we obtain the following diagram 
\begin{equation}
\label{equation-diagram-resolution}\vcenter{
\xy
\xymatrix{
\tilde X^{\sharp}\ar[r]^{\tilde \pi}\ar[d]^{\tilde \sigma}&\tilde X\ar[d]^{\sigma}
\\
X^{\sharp}\ar[r]^{\pi}&X
}
\endxy}
\end{equation}
Here $\sigma: (\tilde X\supset \tilde C)\to (X\ni o)$
is the dlt modification.
\end{construction}

The following definition can be given in arbitrary dimension.
For simplicity we state it only for dimension two which is 
sufficient for our needs.
\begin{case}{\bf Adjunction.}
\label{different}
Let $X$ be a normal surface and $D$ be an effective $\QQ$-divisor on $X$. 
Write $D=C+B$, where $C$ is a reduced divisor on $X$, $B$ is effective,
and $C$ and $B$ have no common component.
Let $\nu:C'\to C$ be the normalization of $C$.
One can construct an effective $\QQ$-divisor $\Diff_C(B)$
on $C'$, called the \textit{different}, as follows;
see \cite[Chap. 16]{Utah} or \cite[\S 3]{Shokurov-1992-e-ba} for details. Take
a resolution of singularities $f:X'\to X$ such that the
proper transform $C'$ of $C$ on $X'$ is also smooth. 
Clearly, $C'$ is nothing but the normalization of the curve $C$.
Let $B'$ be the proper transform of $B$ on $X'$.
One can find an
exceptional $\QQ$-divisor $A$ on $X'$ such that $K_{X'}+C'+B'\equiv_fA$. The different
$\Diff_C(B)$ is defined as the $\QQ$-divisor $(B'-A)|_{ C'}$. Then $\Diff_C(B)$ is effective 
and it satisfies the equality (adjunction formula)
\begin{equation}
\label{equation-adjunction}
K_{C'}+\Diff_C(B)=\nu^* (K_X+C+B)|_C. 
\end{equation}
\end{case}

\begin{stheorem}[Inversion of Adjunction {\cite{Shokurov-1992-e-ba}}, {\cite{Kawakita2007}}]
\label{Inversion-of-Adjunction}
The pair $(X, C+B)$ is lc \textup(resp. plt\textup) near $C$ if and only if
the pair $(C',\Diff_C(B))$ is lc \textup(resp. klt\textup).
\end{stheorem}

\begin{sproposition}
Let $(X\ni P)$ be a surface singularity and let $o\in C \subset X$ be
an effective reduced divisor such that the pair $(X,C)$ is plt.
Then $(P\in C \subset X)$ is analytically isomorphic to
\[
\left(0\in \{x_1-axis\}\subset \CC^2\right)/\mumu_r(1,q),\qquad \gcd(r,q)=1.
\]
In particular, $C$ is smooth at $P$ and $\Diff_C(0)=(1-1/r)P$.
The dual graph of the minimal resolution of $(X\ni P)$ is a chain
\eqref{equation-chain} and the proper transform of $C$ is attached to 
one of its ends.
\end{sproposition}

\section{$\QQ$-Gorenstein smoothings of log canonical singularities}
\label{sect-smoothings}
In this section we prove the classificational part of Theorem~\ref{Theorem-Q-smoothings}.
\begin{notation}
Let $(X\ni P)$ be a normal surface singularity, let $\eta:Y\to X$ be the minimal resolution and
let $E=\sum E_i$ be the exceptional divisor. Write 
\begin{equation}
\label{equation-codiscrepancy}
K_{Y}=\eta ^*K_X-\Delta,
\end{equation}
where $\Delta$ is an effective $\QQ$-divisor with $\Supp(\Delta)=\Supp(E)$.
Thus one can define the self-intersection $\K_{(X,P)}^2:=\Delta^2$ which is a well-defined natural invariant. 
We usually write $\K^2$ instead of $\K_{(X,P)}^2$ if no confusion is likely.
The value $\K^2$ is non-positive and it equals zero if and only if $(X\ni P)$ is a Du Val point.
\begin{itemize}
 \item We denote by $\varsigma_P$ the number of exceptional divisors over $P$.
\end{itemize}
\end{notation}

\begin{lemma}
\label{lemma-canonical}
Let $(X\ni P)$ be a normal surface singularity and let 
$\XX\to \XD$ be its $\QQ$-Gorenstein smoothing.
If $(X\ni P)$ is log terminal,
then the pair $(\XX,X)$ is plt and the singularity $(\XX\ni P)$ 
is terminal.

If $(X\ni P)$ is log canonical,
then the pair $(\XX,X)$ is lc and the singularity $(\XX\ni P)$ 
is isolated canonical.
\end{lemma}

\begin{proof}
By the higher-dimensional version of the inversion of adjunction (see \cite[Th. 5.50]{Kollar-Mori-1988},
\cite{Kawakita2007} and Theorem~\ref{Inversion-of-Adjunction}) 
the singularity $(X\ni P)$ is log terminal \textup(resp. log canonical\textup)
if and only if the pair $(\XX,X)$ is plt \textup(resp. lc\textup) at $P$.
Since $X$ is a Cartier divisor on $\XX$, the assertion follows. 
\end{proof}

\begin{lemma}[{\cite[Proposition 6.2.8]{Kollar1991a}}]
\label{lemma-K2-integer}
Let $(X\ni P)$ be a rational surface singularity.
If $(X\ni P)$ admits a $\QQ$-Gorenstein smoothing, then $\K^2$ is an integer.
\end{lemma}

\begin{theorem}[{\cite[Proposition 3.10]{Kollar-ShB-1988}}, 
{\cite[Proposition 5.9]{Looijenga-Wahl-1986}}]
\label{classification-T-singularities}
Let $(X\ni P)$ be a log terminal surface singularity.
The following are equivalent:
\begin{enumerate}
\item 
$(X\ni P)$ admits a $\QQ$-Gorenstein smoothing,
\item 
$\K^2\in\ZZ$,
\item 
\label{classification-T-singularities-3}
$(X\ni P)$
is either Du Val or a cyclic quotient singularity of the form $\frac{1}{m}(q_1,q_2)$
with 
\begin{equation*}
(q_1+q_2)^2\equiv 0\mod m,\qquad \gcd(m,q_i)=1.
\end{equation*}
\end{enumerate}
\end{theorem}

A log terminal singularity satisfying equivalent conditions above is called a
\textit{$\rT$-singularity}. 

\begin{sremark}[see {\cite{Kollar-ShB-1988}}]
It easily follows from~\ref{classification-T-singularities-3} that
any non-Du Val singularity of type $\rT$ can be written in the form 
\begin{equation*}\textstyle
\frac1{dn^2}(1, dna-1)
\end{equation*}
\end{sremark}
Below we describe log canonical singularities with 
integral $\K^2$. Note however, that in general, the condition $\K^2\in\ZZ$ is necessary but not sufficient 
for the existence of $\QQ$-Gorenstein smoothing (cf. Theorem~\ref{Theorem-Q-smoothings} and Proposition
\ref{Proposition-computation-K2} (DV)).

\begin{proposition}
\label{Proposition-computation-K2}
Let $(X\ni P)$ be a rational strictly log canonical surface singularity. Then 
in the notation of Theorem \xref{theorem-classification-lc-singularities}
the invariant $\K^2$ is integral if and only if
$X$ is either of type $[n_1,\dots, n_s; [2]^4]$ or 
of type $\langle n; r_1,r_2,r_3; \varepsilon,\varepsilon,\varepsilon\rangle$,
where $(r_1,r_2,r_3)=(3,3,3)$, $(2,4,4)$ or $(2,3,6)$ and $\varepsilon=1$ or $-1$.
Moreover, we have:
\begin{itemize}
\item[(DV)] 
if $X$ is of type $[n_1,\dots, n_s; [2]^4]$ or $\langle n; r_1,r_2,r_3; -1,-1,-1\rangle$,
then 
\begin{equation*}
-\K^2=n-2, 
\end{equation*}
where in the case 
$[n_1,\dots, n_s; [2]^4]$, we put $n:=\sum (n_i-2)+2$;
\item[(nDV)] 
if $X$ is of type $\langle n; r_1,r_2,r_3; 1,1,1\rangle$, then 
\begin{equation*}
-\K^2=n-9+\sum r_i.
\end{equation*}
\end{itemize}
\end{proposition}
For the proof we need the following lemma.

\begin{slemma}
\label{Lemma-computation-K2}
Let $V$ be a smooth surface and let $C, E_1,\dots,E_m\subset V$ be 
proper smooth rational curves on $V$ whose configuration is a chain:
\begin{equation*}
\xy
\xymatrix@R=3pt{
\underset{C}\circ&\underset{E_m}\circ\ar@{-}[r]\ar@{-}[l]&\underset{}
\cdots\ar@{-}[r]&\underset{E_1}\circ
}
\endxy
\end{equation*}
Let $D=C+\sum\alpha_i E_i$ be a $\QQ$-divisor such that $(K_V+D)\cdot E_j=0$ for all $j$.
\begin{enumerate}
\item 
If all the $E_i$'s are $(-2)$-curves, then $D^2-C^2=m/(m+1)$. 
\item 
If $m=1$ and $E_1^2=-r$, then $D^2-C^2=(r-1)(3-r)/r$.
\end{enumerate}
\end{slemma}
\begin{proof}
Assume that $E_i^2=-2$ for all $i$.
It is easy to check that 
$D=C+\sum_{i=1}^m\frac i{m+1} E_i$.
Then
\begin{multline*}
D^2-C^2=\frac {2m}{m+1}+\left(\sum_{i=1}^m\frac i{m+1} E_i\right)^2=
\\
=\frac {2m}{m+1}+
\frac2{(m+1)^2}\left(-\sum_{i=1}^m i^2+\sum_{i=1}^{m-1} i(i+1)\right)
=\frac {m}{m+1}.
\end{multline*}

Now let $m=1$ and $E_1^2=-r$. Then $D=C+\frac {r-1}r E_1$.
Hence
\begin{equation*}
D^2-C^2=\frac{2(r-1)}r-\frac{(r-1)^2}{r}=\frac{(r-1)(3-r)}r.\qedhere
\end{equation*}
\end{proof}

\begin{proof}[Proof of Proposition \xref{Proposition-computation-K2}]
Let $\Delta$ be as in \eqref{equation-codiscrepancy}
and let $C:=\lfloor\Delta\rfloor$. 
Write $\Delta=C+\sum\Delta_i$, where $\Delta_i$ are effective connected
$\QQ$-divisors. By Lemma~\ref{Lemma-computation-K2} we have 
\begin{equation*}
\delta_i:=\Bigl(\left(C+\Delta_i\right) ^2-C^2\Bigr)=
\begin{cases}
1-\frac1{r_i}&\text{if $\Delta_i$ is of type $\frac 1{r_i}(1,-1)$,}
\\
4-r_i-\frac3{r_i}&\text{if $\Delta_i$ is of type $\frac 1{r_i}(1,1)$.}
\end{cases}
\end{equation*}
Then
\begin{equation*}
\K^2=\left(C+\sum\Delta_i\right) ^2=C^2+\sum\delta_i.
\end{equation*}
If $(X\ni P)$ is of type $[n_1,\dots,n_s,[2],[2],[2],[2]]$,
then 
\begin{equation*}
\K^2=C^2+2=-\sum (n_i-2). 
\end{equation*}
Assume that $C$ is irreducible and $(X\ni P)$ is of type $\langle n; r_1,r_2,r_3\rangle$,
where $\sum 1/r_i=1$. 

If all the $\Supp(\Delta_i)$'s are Du Val chains, then 
\begin{equation*}
\K^2=C^2+\sum\left(1-\textstyle{\frac{1}{r_i}}\right)=-n+2. 
\end{equation*}
If $(X\ni P)$ is of type 
$\langle n; r_1,r_2,r_3;1,1,1\rangle$, then 
\begin{equation*}
\K^2=C^2+\sum\left(4-r_i-\textstyle{\frac3{r_i}}\right)=-n+9-\sum r_i.
\end{equation*}
It remains to consider the ``mixed'' case. Assume for example that $(X\ni P)$ is of type 
$\langle n; 3,3,3\rangle$.
Then $\delta_i\in\{0,\, 2/3\}$.
Since $\sum\delta_i$ is an integer, the only possibility is
$\delta_1=\delta_2=\delta_3$, i.e. all the chains
$\Delta_i$ are of the same type. The cases $\langle n; 2,4,4\rangle$
and $\langle n; 2,3,6\rangle$ are considered similarly.
\end{proof}

\begin{scorollary}
\label{canonical-cover-Z2}
Let $(X\ni P)$ be a strictly log canonical surface singularity of index $I\ge 2$
admitting a $\QQ$-Gorenstein smoothing. Let $(X^{\sharp}\ni P^{\sharp})\to (X\ni P)$ 
be the index one cover. Then 
\begin{equation*}
-\K^2_{(X^{\sharp}\ni P^{\sharp})}=
\begin{cases}
I(n-2)&\text{in the case \type{(DV)},} 
\\
I(n-1)&\text{in the case \type{(nDV)}.}
\end{cases}
\end{equation*}
\end{scorollary}

\begin{proof}
Let us consider the   \type{(nDV)} case. 
We use the notation of \eqref{equation-min-resolution}
and \eqref{equation-diagram-resolution}. Let $E_1$, $E_2$, $E_3$ be the 
$\tilde\eta$-exceptional divisors. Then
\[
K_{\tilde X}=\sigma^*K_X-\tilde C,\quad
K_Y=\eta^*K_X-\Delta=\tilde \eta^*K_{\tilde X}-\textstyle{\sum \frac {r_i-2}{r_i}E_i}. 
\]
Therefore, 
\[
\Delta=\tilde \eta^* \tilde C+ \textstyle{\sum \frac {r_i-2}{r_i}E_i},
\quad
\Delta^2=\tilde C^2- \textstyle{\sum \left(r_i-4+\frac{4}{r_i}\right)},
\]
\[
 -\tilde C^2= n+3\textstyle{-\sum \frac{4}{r_i}}=n-1,
\quad -\K^2_{(X^{\sharp}\ni P^{\sharp})}=-I\tilde C^2=I(n-1).
\qedhere
\]
\end{proof}

\begin{sremark}
\label{canonical-cover-mult}
In the above notation 
we have (see e.g. \cite[Theorem 4.57]{Kollar-Mori-1988})
\begin{eqnarray*}
\mult (X^{\sharp}\ni P^{\sharp})&=&\max \left(2,-\K^2_{(X^{\sharp}\ni P^{\sharp})}\right),
\\
\operatorname{emb}\dim (X^{\sharp}\ni P^{\sharp})&=&\max \left(3,-\K^2_{(X^{\sharp}\ni P^{\sharp})}\right). 
\end{eqnarray*}
\end{sremark}
The following proposition is the key point in the proof of 
of Theorem~\ref{Theorem-Q-smoothings}.
\begin{proposition}
\label{proposition-nDV}
Let $(X\ni P)$ be a strictly log canonical rational surface singularity of index $I\ge 3$ admitting 
a $\QQ$-Gorenstein smoothing.
Then $(X\ni P)$ is of type $[n; [r_1],[r_2],[r_3]]$. 
\end{proposition}

\begin{proof}
By Lemma~\ref{lemma-K2-integer} the number
$\K^2$ is integral and by Proposition \xref{Proposition-computation-K2}
$(X\ni P)$ is either of type \type{nDV} or of type \type{DV}. 
Assume that $(X\ni P)$ is of type \type{DV}.

\begin{case}
\label{new-label}
Let $\ff: \XX\to\XD$ be a $\QQ$-Gorenstein smoothing.
By Lemma~\ref{lemma-canonical} the pair $(\XX,X)$ is log canonical 
and $(P\in \XX)$ is an isolated canonical singularity. 
Let $\pi: (\XX^\sharp\ni P^\sharp) \to (\XX\ni P)$ be the index one cover (see~\ref{index-one-cover})
and let $X^\sharp:=\pi^*X$. Then $X^\sharp$ is a Cartier divisor on $\XX^\sharp$,
the singularity $(\XX^\sharp\ni P^\sharp)$ is canonical (of index $1$), and the pair 
$(\XX^\sharp, X^\sharp)$ is lc. Moreover, $\XX^\sharp$ is CM, $X^\sharp$ hence normal, and
the canonical divisor $K_{X^\sharp}$ 
is Cartier. Therefore, $\pi$ induces the index one cover $\pi_X: (X^\sharp\ni P^\sharp) \to (X\ni P)$.
In particular, the index of $(P\in \XX)$ equals $I$.
Since $I\ge 3$, the singularity $(X^\sharp\ni P^\sharp)$ is simple elliptic and
the dlt modification coincides with the minimal resolution.
\end{case}

\begin{case}
First we consider the case where $(P\in \XX)$ is \textit{terminal}.
Below we essentially use the classification of terminal singularities (see e.g. \cite{Reid-YPG1987}).
In our case, $(\XX^\sharp\ni P^\sharp)$ is either smooth or an isolated cDV singularity.
In particular,
\begin{equation*}
\operatorname{emb}\dim (X^{\sharp}\ni P^{\sharp})\le \operatorname{emb}\dim (\XX^{\sharp}\ni P^{\sharp})\le 4.
\end{equation*}
By our assumption $(X\ni P)$ is of type \type{DV}. So, 
by Corollary~\ref{canonical-cover-Z2} and Remark~\ref{canonical-cover-mult}
\begin{equation}
\label{equation-emb-dim-I}
\operatorname{emb}\dim (X^{\sharp}\ni P^{\sharp})= I(n-2).
\end{equation}

If $\operatorname{emb}\dim (\XX^{\sharp}\ni P^{\sharp})=3$, i.e. $(\XX^{\sharp}\ni P^{\sharp})$
is smooth, then $\operatorname{emb}\dim (X^{\sharp}\ni P^{\sharp})=3$,
$\mult{(X^{\sharp}\ni P^{\sharp})}=3$, and $I=n=3$.
In this case $(\XX\ni P)$ is a cyclic quotient 
singularity of type $\frac 13(1,1,-1)$ \cite{Reid-YPG1987}.
We may assume that $(\XX^{\sharp},P^\sharp)=(\CC^3,0)$ and $X^{\sharp}$
is given by an invariant equation $\psi(x_1,x_2,x_3)=0$
with $\mult_{0}\psi=3$. Since $(X^{\sharp}\ni P^{\sharp})$ is a simple 
elliptic singularity, the cubic part $\psi_3$ of $\psi$
defines a smooth elliptic curve on $\PP^2$. Hence we can write 
$\psi_3= x_3^3 + \tau(x_1,x_2)$, where $\tau(x_1,x_2)$
is a cubic homogeneous polynomial without multiple factors.
The minimal resolution $\tilde X^\sharp \to X^\sharp$ is the blowup of the origin.
In the affine chart $\{x_2\neq 0\}$ the surface $\tilde X^\sharp$ is given by the 
equation $\tau(x_1',1)+x_3^{\prime 3}+x_2'(\cdots)=0$ and the action 
of $\mumu_3$ is given by the weights $(0,1,1)$. Then it is easy to see that
$\tilde X$ has three singular points of type $\frac13(1,1)$. This contradicts our assumption.

Thus we may assume that $\operatorname{emb}\dim (\XX^{\sharp}\ni P^{\sharp})=4$, i.e. 
$(\XX^{\sharp}\ni P^{\sharp})$ is a hypersurface singularity. Then $I= 4$ by \eqref{equation-emb-dim-I}.
We may assume that $(\XX^{\sharp}\ni P^{\sharp})\subset (\CC^4\ni 0)$ is a hypersurface 
given by an equation $\phi(x_1,\dots,x_4)=0$ with $\mult_0\phi=2$ 
and $X^{\sharp}$ is cut out by an invariant equation $\psi(x_1,\dots,x_4)=0$.
Furthermore, we may assume that $x_1,\dots,x_4$ are semi-invariants with $\mumu_4$-weights $(1,1,-1,b)$,
where $b=0$ or $2$ (see \cite{Reid-YPG1987}).

Consider the case $\mult_0\psi=1$. Since $\psi$ is invariant, we have 
$\psi= x_4+(\text{higher degree terms})$ and $b=0$.
In this case the only quadratic invariants are $x_1x_3$, $x_2x_3$, and $x_4^2$.
Thus $\phi_2$ is a linear combination of $x_1x_3$, $x_2x_3$, $x_4^2$.
Since $I=4$ and $b=0$, by the classification of terminal singularities 
$\phi$ contains either $x_1x_3$ or $x_2x_3$ (see \cite{Reid-YPG1987}). Then the eliminating $x_4$ we see that 
$(X^\sharp\ni P^\sharp)$ is a hypersurface singularity 
whose equation has quadratic part of rank $\ge 2$. 
In this case $(X^\sharp\ni P^\sharp)$ is a Du Val singularity of type \type{A_n}, a contradiction.

Now let $\mult_0\psi>1$. Then 
\begin{equation*}
\operatorname{emb}\dim (X^{\sharp}\ni P^{\sharp})=-K_{(X^{\sharp}\ni P^{\sharp})}
=\mult(X^{\sharp}\ni P^{\sharp})=4=I
\end{equation*}
(see Remark~\ref{canonical-cover-mult}).
According to \cite[Theorem 4.57]{Kollar-Mori-1988} the curve given by quadratic parts 
of $\phi$ and $\psi$ in the projectivization 
$\PP(T_{P^{\sharp},\XX^{\sharp}})$
of the 
tangent space is a smooth elliptic curve.
According to the classification \cite{Reid-YPG1987} there are two cases.

\subsection*{Case: $b=0$ and $\phi$ is an invariant.}
In this case, as above, $\phi_2$ and $\psi_2$ are linear combination of $x_1x_3$, $x_2x_3$, $x_4^2$
and so $\{\phi_2=\psi_2=0\}$ cannot be smooth, a contradiction.

\subsection*{Case: $b=2$ and $\phi$ is a semi-invariant of weight $2$.}
Then, up to linear coordinate change of $x_1$ and $x_2$, we can write 
\begin{equation*}
\phi_2=a_1 x_1x_2+a_2 x_1^2+a_3 x_2^2+a_4 x_3^2,\qquad 
\psi_2=b_1x_1x_3+b_2x_2x_3+b_3 x_4^2. 
\end{equation*}
Since $\phi_2=\psi_2=0$ defines a smooth curve,
$a_1 x_1x_2+a_2 x_1^2+a_3 x_2^2$ has no multiple factors, so up to linear coordinate change of $x_1$ and $x_2$
we may assume that 
$\phi_2=x_1x_2+ x_3^2$.
Similarly, $b_1,\, b_2,\, b_3\neq 0$. Then easy computations (see e.g \cite[7.7.1]{Kollar-Mori-1992}) show that 
$(X^\sharp\ni P^\sharp)$ is a singularity of type $[2;[2],[4]^2]$.
This contradicts our assumption.
\end{case}

\begin{case}
Now we assume that $(P\in \XX)$ is \textit{strictly canonical}.
Let $\upgamma : \tilde \XX\to \XX$ be the \textit{crepant blowup} of $(P\in \XX)$.
By definition $\tilde \XX$ has only $\QQ$-factorial terminal singularities
and $K_{\tilde \XX}=\upgamma^* K_{\XX}$.
Let $E=\sum E_i$ be the exceptional divisor and let $\tilde X$
be the proper transform of $X$. Since the pair $(\XX,X)$ is log canonical,
we can write 
\begin{equation}
\label{equation-discrepancies-deformation-space}
K_{\tilde \XX}+\tilde X+E=\upgamma^* (K_{\XX}+X),\qquad \upgamma^* X=\tilde X+E.
\end{equation}
The pair $(\tilde \XX, \tilde X+E)$ is log canonical
and $\tilde \XX$ has isolated singularities, so $\tilde X+E$ has generically 
normal crossings along $\tilde X\cap E$. Hence $C:=\tilde X\cap E$ is a reduced curve.
By the adjunction we have 
\begin{equation*}
K_{\tilde X}+C= (K_{\tilde \XX}+\tilde X+E)|_{\tilde X}= \upgamma^* (K_{\XX}+X)|_{\tilde X}
=\upgamma_{\tilde X}^* K_{X}.
\end{equation*}
Thus $\upgamma_{\tilde X}: \tilde X\to X$ is a dlt modification of $(X\ni P)$.
Since $I\ge 3$, there is only one divisor over $P\in X$ with discrepancy $-1$.
Hence this divisor coincides with $C$ and so $C$ is irreducible and smooth.
In particular, $\tilde X$ meets only one component of $E$. 

\begin{claim*}
Let $Q\in \tilde \XX$ be a point at which $E$ is not Cartier.
Then in a neighborhood of $Q$ we have $\tilde X\sim K_{\tilde \XX}$.
In particular, $Q\in C$.
\end{claim*}

\begin{proof}
We are going to apply the results of \cite{Kawakita-index}.
The extraction $\upgamma: \tilde \XX\to \XX$ can be decomposed in a sequence 
of elementary crepant blowups 
\begin{equation*}
\upgamma_i: \XX_{i+1} \longrightarrow \XX_{i}, \quad i=0,\dots, N,
\end{equation*}
where $\XX_0=\XX$, $\XX_{N}=\tilde \XX$, for $i=1,\dots, N$
each $\XX_{i}$ has only $\QQ$-factorial canonical singularities,
and the $\upgamma_i$-exceptional divisor $E_{i+1,i}$ is irreducible.
\cite{Kawakita-index} defined a divisor $F$ with $\Supp(F)=E$ on $\XX_{N}=\tilde \XX$ 
inductively: $F_1=E_{1,0}$ on $\XX_1$ and $F_{i+1}= \lceil \upgamma_i^*F_i\rceil$.
In our case, by \eqref{equation-discrepancies-deformation-space} the divisor $F$ is reduced, i.e. 
$F=E$. Then by \cite[Theorem 4.2]{Kawakita-index} we have $E\sim -K_{\tilde \XX}$
near $Q$. Since $\tilde X+E$ is Cartier, 
$\tilde X\sim K_{\tilde \XX}$
near $Q$.
\end{proof}

\begin{claim*} 
The singular locus of $\tilde \XX$ near $C$ consists of three 
cyclic quotient singularities $P_1$, $P_2$, $P_3$ 
of types $\frac 1{r_i}(1,-1, b_i)$, where $\gcd (b_i, r_i)=1$ and 
$(r_1,r_2,r_3)=(3,3,3)$, $(2,4,4)$, and $(2,3,6)$ in cases $I=3$, $4$, $6$,
respectively.
\end{claim*}
\begin{proof}
Let $P_1,\, P_2,\, P_3\in C$ be singular points 
of $\tilde X$. 
Since $C=\tilde X\cap E$ is smooth, $E$ is not Cartier at $P_i$'s. 
Hence $P_1,\, P_2,\, P_3\in \tilde \XX$ are (terminal) non-Gorenstein points.
Now the 
assertion follows by \cite[Theorem 4.2]{Kawakita-index}.
\end{proof}
Therefore, 
$P_i\in \tilde X$ is a point of index $r_i/\gcd(2,r_i)$. Hence the 
singularities of $\tilde X$ are of types $\frac 1{r_i}(1,1)$.
This proves Proposition~\ref{proposition-nDV}.\qedhere
\end{case}
\end{proof}

\begin{case}
Let $(X\ni P)$ be a normal surface singularity admitting a $\QQ$-Gorenstein smoothing
$\ff: \XX\to\XD$. Let $M_P$ be the Milnor fiber of $\ff$. Thus,
$(M_P,\partial M_P )$ is a smooth 4-manifold with boundary. 
Denote by $\upmu_P = \bb_2 (M_P)$ the Milnor number of the smoothing.
In our case we have (see \cite{Greuel-Steenbrink-1983})
\begin{equation}
\label{equation-Milnor-fiber-muP}
\bb_1 (M_P) = 0,\qquad \Eu(M_P) = 1 + \upmu_ P.
\end{equation}
\end{case}

\begin{sproposition}[cf. {\cite[\S 2.3]{Hacking-Prokhorov-2010}}]
\label{proposition-computation-muP}
Let $(X\ni P)$ be a rational surface singularity. 
Assume that $(X\ni P)$ admits a $\QQ$-Gorenstein smoothing. Then
for the Milnor number
$\upmu_P$ we have
\begin{equation}
\label{equation-computation-muP}
\upmu_P=\K_{(X,P)}^2+\varsigma_P. 
\end{equation}
\end{sproposition}

\begin{proof}
Obviously, $\K_{(X,P)}^2+\varsigma_P$ depends only on the analytic type of the singularity $(X\ni P)$.
According to \cite[Appendix]{Looijenga1986}, for $(X\ni P)$ there exists a projective surface 
$Z$ with a unique singularity isomorphic to $(X\ni P)$ and a $\QQ$-Gorenstein smoothing
$\mathfrak Z/ (\mathfrak T\ni 0)$. 
Let $\eta:Y\to Z$ be the minimal resolution.
Write 
\begin{equation*}
K_{Y}=\eta^* K_Z-\Delta,\qquad 
K_{Y}^2=K_Z^2+\Delta^2. 
\end{equation*}
Let $Z'$ be the general fiber. Since 
\begin{equation*}
\Eu(Y)=\Eu(Z)+\varsigma_P,\qquad 
\chi(\OOO_{Y})=\chi(\OOO_Z),
\end{equation*}
by Noether's formula we have
\begin{multline*}
0=K_Y^2+\Eu(Y)-12\chi(\OOO_Z)=K_Z^2+\Delta^2+\Eu(Z)+\varsigma_P-12\chi(\OOO_{Z'})=
\\
=\Delta^2+\varsigma_P+\Eu(Z)+K_{Z'}^2-12\chi(\OOO_{Z'})=
\Delta^2+\varsigma_P+\Eu(Z)-\Eu(Z').
\end{multline*}
By \eqref{equation-Milnor-fiber-muP} we have 
$\upmu_P=\Delta_P^2+\varsigma_P$.
\end{proof}

\begin{scorollary}[see {\cite[Proposition 13]{Manetti-1991}}]
\label{Corollary-computation-muP-T}
If $(X\ni P)$ is a $\rT$-singularity of type $\frac1{dm^2}(1, dma-1)$, then 
\begin{equation}
\label{equation-computation-muP-T}
\upmu_P=d-1,\qquad-\K^2=\varsigma_P-d+1.
\end{equation}
\end{scorollary}

Proposition~\ref{proposition-computation-muP} implies the following.

\begin{scorollary}
\label{Corollary-mu-2}
Let $(X\ni P)$ be a strictly log canonical surface singularity of index $I>1$
admitting a $\QQ$-Gorenstein smoothing. Then 
\begin{equation}
\label{equation-computation-muP-nG}
\upmu_P=
\begin{cases}
4-\sum (n_i-3)&\text{in the case $(\mathrm{DV})$ with $I=2$,}
\\
13-n-\sum r_i&\text{in the case $(\mathrm{nDV})$.}
\end{cases}
\end{equation}
\end{scorollary}

\begin{proof}[Proof of the classificational part of Theorem \xref{Theorem-Q-smoothings}]
Let 
\[
\pi:(X^{\sharp}\ni P^{\sharp})\to (X\ni P)
\]
be the index one cover.
A $\QQ$-Gorenstein smoothing $(X\ni P)$ is induced by an equivariant
smoothing of $(X^{\sharp}\ni P^{\sharp})$ (cf.~\ref{new-label}). 
In particular, $(X^{\sharp}\ni P^{\sharp})$ is smoothable.
Assume that $(X\ni P)$ is of type $[n_1,\dots,n_s;[2]^4]$ with $s>1$. Then
$(X^{\sharp}\ni P^{\sharp})$ is a cusp singularity. By \cite[Th. 5.6]{Wahl-1981} its smoothability implies 
\begin{equation*}
\mult (X^{\sharp}\ni P^{\sharp})\le\varsigma_{P^{\sharp}}+9.
\end{equation*}
Since $\varsigma_{P^{\sharp}}=2\varsigma_{P}-10$,
by Corollary~\ref{canonical-cover-Z2} and Remark~\ref{canonical-cover-mult}
we have
\begin{equation*}
2\sum\left(n_i-2\right)\le 2\varsigma_{P}-1,\quad\sum\left(n_i-3\right)\le 3.
\end{equation*}
In the case where $(X\ni P)$ is of type $[n;[2]^4]$ the singularity $(X^{\sharp}\ni P^{\sharp})$ is
simple elliptic. Then $\mult (X^{\sharp}\ni P^{\sharp})\le 9$ (see e.g. \cite[Ex. 6.4]{Looijenga-Wahl-1986}).
Hence $n\le 6$. 
In the case where $(X^{\sharp}\ni P^{\sharp})$ is of type $[n,[r_1], [r_2], [r_3]]$
the assertion follows from Corollary~\ref{Corollary-mu-2} because $\upmu_P\ge 0$. 
\end{proof}
The existence of $\QQ$-Gorenstein smoothings follows from 
examples and discussions in the next two sections.

\section{Examples of $\QQ$-Gorenstein smoothings}
\label{section-Examples}
\begin{proposition}[{\cite[Cor. 19]{Stevens1991a}}]
A rational surface singularity of index $2$ and multiplicity $4$ admits a $\QQ$-Gorenstein 
smoothing.
\end{proposition}
Recall that for any rational surface singularity $(X\ni P)$ 
one has 
\[
\mult (X\ni P)=-\mathcal{Z}^2,
\]
where $\mathcal{Z}$ is the fundamental cycle on the minimal resolution
(see \cite[Cor. 6]{Artin-1966}). 
\begin{slemma}
Let $(X\ni P)$ be  a log canonical surface singularity  of type 
$[n_1,\dots,n_s; [2]^4]$.
Then 
\begin{equation*}
-\mathcal{Z}^2=\max \bigl(4, 2+\sum (n_i-2)\bigr)=\max\bigl(4, 2-\K^2\bigr).
\end{equation*}
\end{slemma}

\begin{proof}
If either $s\ge 2$ and $n_1,\, n_s\ge 3$ or $s=1$ and $n_1\ge 4$, then 
$\mathcal{Z}=\lceil\Delta\rceil$ and so $\mathcal{Z}^2=\Delta^2-2=-n$
by Proposition~\ref{Proposition-computation-K2}. If $\sum (n_i-2)=1$, 
then $\mathcal{Z}=2\Delta$ and so 
$\mathcal{Z}^2=4\Delta^2=-4$.
\end{proof}

\begin{scorollary}
A log canonical singularity of type $[n_1,\dots, n_s; [2]^4]$
with $\sum (n_i-2)\le 2$ admits a $\QQ$-Gorenstein 
smoothing.
\end{scorollary}

Let us consider explicit examples.

\begin{sexample}
\label{example-singularity-index-2-m}
Let $\XX=\CC^3/\mumu_2(1,1,1)$ and 
\begin{equation*}
\ff: \XX \to \CC,\quad (x_1,x_2,x_3) \mapsto x_1^2+\left(x_2^2+c_1 x_3^{2k}\right)
\left(x_3^2+c_2 x_2^{2m}\right),
\end{equation*}
where $k, m\ge 1$ and $c_1,\, c_2$ are constants.
The central fiber $X=\XX_0$ is a log canonical singularity 
of type 
\begin{equation*}
[\underbrace{2,\dots,2}_{k-1},3,\underbrace{2,\dots,2}_{m-1}; [2]^4].
\end{equation*}
Indeed, the $\frac12(1,1,1)$-blowup of $X'\to X\ni 0$ has irreducible exceptional divisor.
If $k,m \ge 3$, then 
the singular locus if $X'$ consists of two
Du Val singularities of types \type{D_{k+1}} and \type{D_{m+1}}. Other cases are similar.
\end{sexample}

\begin{sexample}
\label{example-2}
Let $\mumu_2$ act on $\CC^4_{x_1,\dots,x_4}$ diagonally 
with weights $(1,1,1,0)$ and let $\phi(x_1,\dots,x_4)$ and $\psi(x_1,\dots,x_4)$
be invariants such that $\mult_0\phi=\mult_0\psi=2$ and the quadratic parts
$\phi_{(2)}$, $\psi_{(2)}$ define a smooth elliptic curve in $\PP^3$.
Let $\XX:=\{\phi=0\}/\mumu_2(1,1,1,0)$. 
Consider the family
\begin{equation*}
\ff: \XX \longrightarrow\CC \qquad (x_1,\dots,x_4)\longmapsto \psi. 
\end{equation*}
The central fiber $X=\XX_0$ is a log canonical singularity 
of type $[4; [2]^4]$. 
\end{sexample}

\begin{sproposition}[{\cite[Ex. 4.2]{deJong-vanStraten-1992}}]
Singularities of types $[5; [2]^4]$,
$[4,3; [2]^4]$, and $[3,3,3; [2]^4]$ admit $\QQ$-Gorenstein smoothings.
\end{sproposition}

Now consider singularities of index $>2$.

\begin{example}[cf. {\cite[6.7.1]{Kollar-Mori-1992}}]
\label{example-singularity-index-3}
Let $\XX=\CC^3/\mumu_3(1,1,2)$ and 
\begin{equation*}
\ff: (x_1,x_2,x_3) \longmapsto x_1^3+x_2^3+x_3^3.
\end{equation*}
The central fiber $X=\XX_0$ is a log canonical singularity 
of type $[2; [3]^3]$.
\end{example}

\begin{example}
\label{example-singularity-index-3-canonical}
Let $\XX=\CC^3/\mumu_9(1,4,7)$ and
\begin{equation*}
\ff: (x_1,x_2,x_3) \longmapsto x_1x_2^2+x_2x_3^2+x_3x_1^2.
\end{equation*}
The central fiber $X=\XX_0$ is a log canonical singularity of type $[4; [3]^3]$. 
The total space has a canonical singularity at the origin.
\end{example}

\begin{example}[cf. {\cite[7.7.1]{Kollar-Mori-1992}}]
\label{example-4}
\label{example-singularity-index-4}
Let 
\begin{equation*}
\XX=\{x_1x_2+x_3^2+x_4^{2k+1}=0\}/\mumu_4(1,1,-1,2),\qquad k\ge 1. 
\end{equation*}
Consider the family
\begin{equation*}
\ff: \XX \longrightarrow\CC,\quad 
(x_1,\dots,x_4)\longmapsto x_4^2+x_3(x_1+x_2)+\psi_{\ge 3} (x_1,\dots,x_4),
\end{equation*}
where $\psi_{\ge 3}$ is an invariant with $\mult(\psi_{\ge 3})\ge 3$.
The central fiber $X=\XX_0$ is a log canonical singularity 
of type $[2; [2], [4]^2]$.
The singularity of the total space is terminal of type $\mathrm{cAx/4}$.
\end{example}

\begin{example}
\label{example-singularity-index-4-canonical}
Let $\XX:=\{x_1x_2+x_3^2+x_4^2=0\}/ \mumu_8(1,5,3,7)$.
Consider the family
\begin{equation*}
\ff: \XX \longrightarrow\CC,\quad (x_1,\dots,x_4)\longmapsto x_1x_4+x_2x_3.
\end{equation*}
The central fiber $X=\XX_0$ is a log canonical singularity 
of type $[3; [2], [4]^2]$.
The singularity of the total 
space $\XX$ is canonical \cite{Hayakawa-Takeuchi-1987}.
\end{example}

More examples of $\QQ$-Gorenstein smoothings will be given in the next section.

\section{Indices of canonical singularities}
\label{sect-Indices}
\begin{notation}
Let $S=S_d\subset \PP^d$ be a smooth del Pezzo surface of degree $d\ge 3$.
Let $Z$ be the affine cone over $S$ and let $z\in Z$ be its vertex. 
Let $\delta: \tilde Z\to Z$ be the blowup along the maximal ideal of $z$ and let $\tilde S\subset \tilde Z$
be the exceptional divisor. 
The affine variety $Z$ can be viewed as the spectrum of the anti-canonical graded algebra:
\begin{equation*}
Z=\Spec R(-K_S), \qquad R(-K_S):= \bigoplus_{n\ge 0} H^0(S, \OOO_S(-nK_S))
\end{equation*}
and the variety $\tilde Z$ can be viewed as the total space $\Tot(\LLL)$
of the line bundle $\LLL:=\OOO_S(K_S)$. Here $\tilde S$ is the negative section.
Denote by $\gamma: \tilde Z \to S$ the natural projection. 
\end{notation}

\begin{lemma}
\label{lemma-crepant}
The map $\delta$ is a crepant morphism and
$(Z\ni z)$ is a canonical singularity.
\end{lemma}

\begin{proof}
Write $K_{\tilde Z}=\delta^* K_Z+a \tilde S$.
Then
\begin{equation*}
K_{\tilde S}= (K_{\tilde Z}+\tilde S) |_{\tilde S}= (a+1)\tilde S|_{\tilde S}.
\end{equation*}
Under the natural identification $S=\tilde S$ one has 
$\OOO_{\tilde S}(K_{\tilde S})\simeq \OOO_{S}(-1)\simeq \OOO_{\tilde S}(\tilde S)$. Hence, $a=0$.
\end{proof}

\begin{construction}
Assume that $S$ admits an action $\varsigma: G\to \Aut(S)$ of a finite group $G$.
The action naturally extends to
an action on the algebra $R(-K_S)$, the cone $Z$, and its blowup $\tilde Z$.
We assume that 
\begin{enumerate}
\renewcommand\labelenumi{{\rm (\Alph{enumi})}}
\renewcommand\theenumi{{\rm (\Alph{enumi})}}
\item 
\label{condition-1}
$G\simeq \mumu_I$ is a cyclic group of order $I$,
\item
\label{condition-2}
the action $G$ on $S$ is free in codimension one, and 
\item
\label{condition-3}
the quotient 
$S/G$ has only Du Val singularities. 
\end{enumerate}

Let $G_P$ be the stabilizer of a point $P\in S$. Since $\LLL=\OOO_S(K_S)$, 
the fiber $\LLL_P$ of $\gamma: \tilde Z=\Tot(\LLL)\to S$ is naturally identified with 
$\wedge ^2 T_{P,S}^\vee$, where $T_{P,S}$ is the tangent space to $S$ at $P$.
By our assumptions~\ref{condition-2} and~\ref{condition-3}, in suitable analytic coordinates $x_1,x_2$ near $P$, the action 
of $G_P$ is given by 
\begin{equation}
\label{equation-DuVal-action}
(x_1,\, x_2) \longmapsto (\zeta_{I_P}^{b_P}\cdot x_1,\, \zeta_{I_P}^{-b_P}\cdot x_2), 
\end{equation}
where $\zeta_{I_P}$ is a primitive $I_P$-th root of unity,
$\gcd(I_P, b_P)=1$, and
$I_P$ is the order of $G_P$. Therefore,
the action of $G_P$ on $\LLL_P\simeq \wedge ^2 T_{P,S}^\vee$
is trivial. Let $\tilde P:= \LLL_P\cap \tilde S$.

The algebra $R(-K_S)$ admits also a natural $\CC^*$-action 
compatible with the grading. Thus $\gamma: \tilde Z\to S$ is a 
$\CC^*$-equivariant $\mathbb A^1$-bundle, where $\CC^*$-action on $S$ is trivial
and 
the induced action $\lambda: \CC^*\to \Aut (\tilde Z)$ is just 
multiplication in fibers. Fix an embedding $G=\mumu_I\subset \CC^*$.
Then two actions $\varsigma$ and $\lambda$ commute and so we can
define a new action of $G$ on $\tilde Z$ by 
\begin{equation} 
\label{equation-index-character}
\varsigma'(\upalpha)=\lambda (\upalpha) \varsigma(\upalpha), \qquad \upalpha\in G.
\end{equation}
Take local coordinates 
$x_1,x_2,x_3$ in a neighborhood of $\tilde P\in \tilde Z$ compatible with the decomposition 
$T_{\tilde P, \tilde Z}=T_{\tilde P, \tilde S}\oplus T_{\tilde P, \LLL_P}$ of the tangent space 
and \eqref{equation-DuVal-action}.
Then
the action of $G_P$ is given by 
\begin{equation}
\label{equation-terminal-action}
(x_1, x_2, x_3) \longmapsto (\zeta_{I_P}^{b_P}\cdot x_1, \zeta_{I_P}^{-b_P}\cdot x_2, 
\zeta_{I_P}^{a_P}\cdot x_3),\quad \gcd(a_P,I_P)=1.
\end{equation}

\begin{claim}
\label{claim-canonical-terminal}
The quotient 
$\tilde \XX:=\tilde Z/ \varsigma'(G)$ has only terminal singularities.
\end{claim}
\begin{proof}
All the points of $\tilde Z$ with non-trivial stabilizers lie on the negative section $\tilde S$.
The image of such a point $\tilde P$ on $\tilde \XX$ is 
a cyclic quotient singularity of type $\frac 1{I_P}(b_P,-b_P, a_P)$ by \eqref{equation-terminal-action}.
\end{proof}

By the universal property of quotients, 
there is a contraction $\varphi: \tilde \XX\to \XX$ contracting $E$ to a point, say $o$,
where $\XX:=Z/G$ and $E:=\tilde S/G$.
Thus we have the following diagram:
\begin{equation}
\label{equation-diagram-square}
\vcenter{
\xymatrix@C=10pt
{
&&\tilde S \ar@{}[r]|-*[@]{\subset}\ar[d] & \tilde Z\ar[d]^\delta\ar[rr]\ar @/^/ [dlll]_(.6){\gamma}|(.43)\hole&
& \tilde \XX\ar[d]^{\varphi}\ar@{}[r]|-*[@]{\supset}& E \ar[d]
\\
S&&z\ar@{}[r]|-*[@]{\in} &Z\ar[rr]^{\pi} && \XX\ar@{}[r]|-*[@]{\ni}& o
}}
\end{equation}
\end{construction}

\begin{proposition}
\label{proposition-main-construction}
$(\XX\ni o)$ is an isolated canonical non-terminal singularity of index $|G|$.
\end{proposition}
\begin{proof}
Since the action $\varsigma'$ is free in codimension one, 
the contraction $\varphi$ is crepant by Lemma~\ref{lemma-crepant}. The index of $(\XX\ni o)$
is equal to the l.c.m. 
of $|G_P|$ for $P\in S$. On the other hand, by the holomorphic Lefschetz fixed point formula
$G$ has a fixed point on $S$. Hence, $G=G_P$ for some $P$.
\end{proof}

\begin{case}
\label{examples-index}
Now we construct explicit examples of del Pezzo surfaces 
with cyclic group actions satisfying the conditions~\ref{condition-1}-\ref{condition-3}.
\end{case}

\begin{sexample}
\label{example-deg=6}
Recall that a del Pezzo surface $S$ of degree $6$ is 
unique up to isomorphism and can be given in $\PP^1_{u_0:u_1}\times \PP^1_{v_0: v_1}\times \PP^1_{w_0: w_1}$ 
by the equation
\begin{equation*}
u_1v_1w_1 =u_0v_0w_0. 
\end{equation*}
Let $\upalpha\in \Aut(S)$ be the following element of order $6$:
\begin{equation*}
\upalpha: (u_0:u_1;\, v_0:v_1;\, w_0: w_1) \longmapsto (v_1:v_0;\, w_1:w_0;\, u_1:u_0).
\end{equation*}
Points with non-trivial stabilizers belong to one of three orbits
and representatives are the following:
\begin{itemize}
\item 
$P=(1:1;\ 1:1;\ 1:1)$,\quad $|G_{P}|=6$, 
\item
$Q=(1:\zeta_3;\ 1:\zeta_3;\ 1:\zeta_3)$,\quad $|G_{Q}|=3$,
\item
$R=(1:1;\ 1:-1;\ 1:-1)$,\quad $|G_{R}|=2$. 
\end{itemize}
It is easy to check that they give us Du Val points of type 
\type{A_5}, \type{A_2}, \type{A_1}, respectively.
\end{sexample}

\begin{sexample}
\label{example-deg=5}
A del Pezzo surface $S$ of degree $5$ is obtained by blowing up four 
points $P_1$, $P_2$, $P_3$, $P_4$ on $\PP^2$ in general position.
We may assume that $P_1 = (1: 0: 0)$, $P_2 = (0: 1: 0)$, $P_3 = (0: 0: 1)$, $P_4 = (1: 1: 1)$. 
Consider the following Cremona transformation:
\begin{equation*}
\upalpha: (u_ 0 : u_ 1 : u_ 2 ) \longmapsto (u _0 (u _2 - u_ 1 ): u _2 (u_ 0 - u _1 ): u _0 u_ 2 ).
\end{equation*}
It is easy to check that $\upalpha^5=\operatorname{id}$ and
the indeterminacy points are exactly $P_1$, $P_2$, $P_3$.
Thus $\upalpha$ lifts to an element $\upalpha\in \Aut(S)$ of order $5$.
\begin{claim*}
Let $\upalpha\in \Aut(S)$ be any element of order $5$.
Then $\upalpha$ has only isolated 
fixed points and the singular locus of the quotient $S/\langle a\rangle$ 
consists of two Du Val points of type \type{A_4}.
\end{claim*}
\begin{proof}
For the characteristic polynomial of $\upalpha$ on $\Pic(S)$ 
there is only one possibility: $t^5-1$. Therefore, the eigenvalues of $\upalpha$ are
$1,\zeta_5,\dots, \zeta_5^4$. This implies that every invariant curve is linearly proportional (in $\Pic(S)$) to 
$-K_S$. In particular, this curve must be an ample divisor.

Assume that there is a curve of fixed points.
By the above it meets any line.
Since on $S$ there are at most two lines passing through a fixed point,
all the lines must be invariant. In this case $\upalpha$ 
acts on $S$ identically, a contradiction.

Thus the action of $\upalpha$ on $S$ is free in codimension one.
By the topological Lefschetz fixed point formula 
$\upalpha$ has exactly two fixed points, say $Q_1$ and $Q_2$.
We may assume that actions of $\upalpha$ in local coordinates near $Q_1$ and $Q_2$
are diagonal:
\begin{equation*}
(x_1,x_2) \longmapsto (\zeta_5^r x_1, \zeta_5^k x_2),\qquad 
(y_1,y_2) \longmapsto (\zeta_5^l y_1, \zeta_5^m y_2),
\end{equation*}
where $r$, $k$, $l$, $m$ are not divisible by $5$.
Then by the holomorphic Lefschetz fixed point formula
\begin{equation*}
1= (1-\zeta_5^r)^{-1}(1-\zeta_5^k)^{-1} + (1-\zeta_5^l)^{-1}(1-\zeta_5^m)^{-1}. 
\end{equation*}
Easy computations with cyclotomics show that up to permutations and modulo $5$
there is only one possibility: $r=1$, $k=4$, $l=2$, $m=3$.
This means that the quotient has only Du Val singularities of type \type {A_4}.
\end{proof}
\end{sexample}

\begin{sexample}
\label{example-P2} 
Let $\mumu_3$ act on $S=\PP^2$ diagonally with weights $(0,1,2)$.
The quotient has three Du Val singularities of type \type{A_2}.
\end{sexample}

\begin{sexample}
\label{example-P1P1}
Let $\mumu_4$ act on $S=\PP^1_{u_0:u_1}\times \PP^1_{v_0: v_1}$ by 
\begin{equation*}
(u_0:u_1;\ v_0: v_1) \longmapsto (v_0: v_1;\ u_1:u_0 ).
\end{equation*}
The quotient has three Du Val singularities of types \type{A_1}, \type{A_3}, \type{A_3}.
\end{sexample}

Note that in all examples above the group generated by $\upalpha^n$ 
also satisfies the conditions~\ref{condition-1}-\ref{condition-3}.
We summarize the above information in the following table.
Together with Proposition~\ref{proposition-main-construction} this proves Theorem~\ref{theorem-main-index}.
\par\medskip\noindent
\setlength{\extrarowheight}{6pt}
\begin{tabularx}{\textwidth}{c|c|c|c|c|l}
No. & $K_S^2$ & Ref. & $G$ & $I$ & \multicolumn1c{$\operatorname{Sing}(\tilde \XX)$} 
\\
\hline
\nrr
\label{smoothing-index=6}& $6$&~\ref{example-deg=6}& $\langle\upalpha\rangle$ & $6$ & 
$\frac16(1,-1,1)$, $\frac13(1,-1,1)$, $\frac12(1,1,1)$
\\
\nrr
\label{smoothing-index=5}&$5$&~\ref{example-deg=5}& $\langle\upalpha\rangle$ & $5$ & 
$\frac15(1,-1,1)$, $\frac15(2,-3,1)$
\\
\nrr
\label{smoothing-index=4}&$8$&~\ref{example-P1P1}& $\langle\upalpha\rangle$ & $4$ & 
$2\times \frac14(1,-1,1)$, $\frac12(1,1,1)$
\\
\nrr
\label{smoothing-index=3-deg=6}& $6$&\ref{example-deg=6}& $\langle\upalpha^2\rangle$ & $3$ & 
$3\times \frac13(1,-1,1)$
\\
\nrr
\label{smoothing-index=3-deg=9}& $9$&\ref{example-P2}& $\langle\upalpha\rangle$ & $3$ & 
$3\times \frac13(1,-1,1)$
\\
\nrr
\label{smoothing-index=2-deg=6}& $6$&\ref{example-deg=6}& $\langle\upalpha^3\rangle$ & $2$ & 
$4\times \frac12(1,1,1)$
\\
\nrr
\label{smoothing-index=2-deg=8}& $8$&\ref{example-P1P1}& $\langle\upalpha^2\rangle$ & $2$ & 
$4\times \frac12(1,1,1)$

\end{tabularx}
\par\medskip\noindent
Note that our table agrees with the corresponding one in \cite{Kawakita-index}.

Now we apply the above technique to construct examples of $\QQ$-Gorenstein smoothings.

\begin{theorem}
\label{theorem-Q-smmoothings-6}
Let $(X\ni o)$ be a surface log canonical singularity 
of one of the following types
\begin{equation*}
[2; [2,3,6]],\hspace{5pt} [3; [2,4,4]],\hspace{5pt} 
[n; [3,3,3]],\hspace{3pt} n=3,4,\hspace{5pt} [n; [2,2,2,2]],\hspace{3pt} n=5,6. 
\end{equation*}
Then $(X\ni o)$ admits a $\QQ$-Gorenstein smoothing.
\end{theorem}

\begin{slemma}
\label{lemma-index}
In the notation of \eqref{equation-diagram-square}, 
let $C\subset S$ be a smooth elliptic $G$-invariant curve such that 
$C\sim -K_S$. Assume that $C$ 
passes through all the points with non-trivial stabilizers. 
Let $\tilde X^\sharp:= \gamma^{-1}(C)$, $X^\sharp:= \delta(\tilde X^\sharp)$,
and $X:=\pi(X^\sharp)$. 
Then the singularity $(X\ni o)$ is log canonical of index $|G|$. Moreover, 
replacing $\lambda$ with $\lambda^{-1}$ if necessary we may assume that
$X$ is a Cartier divisor on $\XX$.
\end{slemma}
\begin{proof}
Put $\tilde X:= \tilde X^\sharp/G$. Since the divisor $\tilde X^\sharp+\tilde S$ 
is trivial on $\tilde S$, the contraction $\delta$ is log crepant with respect 
to $K_{\tilde Z}+\tilde X^\sharp+\tilde S$ and so $\varphi$ is with respect to 
$K_{\tilde \XX}+\tilde X+E$. By construction $X^\sharp$ is a cone over the 
elliptic curve $C$ and $X=X^\sharp/G$. Therefore, $(X\ni o)$ is a log canonical 
singularity. Comparing with~\ref{construction-index-one-cove-log-canonical} we 
see that the index of $(X\ni o)$ equals $|G|$. We claim that $\tilde X+E$ is a 
Cartier divisor on $\tilde \XX$. Identify $C$ with $\tilde C:=\gamma^{-1}(C)\cap \tilde 
S=\tilde S\cap \tilde X^\sharp$.

Let $\omega\in H^0(C,\OOO_C(K_C))$ be a nowhere vanishing holomorphic $1$-form on $C$
and let $\upalpha$ be a generator of $G$. Since $\dim H^0(C,\OOO_C(K_C))=1$ and $G$ 
has a fixed point on $C$, the action of $G$ on  $H^0(C,\OOO_C(K_C))$ is faithful and
we can write $\upalpha^*\omega=\zeta_I \omega$, where $\zeta_I$ is a suitable  primitive 
$I$-th root of unity.  

Pick a point $\tilde P\in \tilde Z$ with 
non-trivial stabilizer $G_P$ of order $I_P$. By our assumptions $\tilde P\in \tilde C$. Take 
semi-invariant local coordinates $x_1,x_2,x_3$ as in \eqref{equation-terminal-action}. 
Moreover, we can take them so that $x_1$ is a local coordinate along $C$. 
Then we can write $\omega=\varpi d x_1$, where $\varpi$ is an invertible 
holomorphic function in a neighborhood of $P$. Hence,  $\varpi$ is an
invariant and $\upalpha^*x_1=\zeta_I^{I/I_P} x_1$.
Thus, by 
\eqref{equation-terminal-action}, the action near $\tilde P$ has the form $\frac 
1{I_P}(1,-1,a_P)$.  
Since $G$ faithfully acts on $C$ with a fixed point, $I_P=2$, $3$, $4$ or $6$. 
Since $\gcd(a_P,I_P)=1$, we 
have $a_P\in \{\pm 1\}$. 
Then by \eqref{equation-index-character} replacing $\lambda$ with $\lambda^{-1}$ 
we may assume that $a_P=1$. In our coordinates the local equation of $\tilde 
S$ is $x_3=0$ and the local equation of $\tilde X^\sharp$ is $x_2=0$. Now it is 
easy to see that the local equation $x_2x_3=0$ of $\tilde S+\tilde X^\sharp$ is 
$G_P$-invariant. Therefore, $\tilde X+E$ is Cartier. Since it is 
$\varphi$-trivial, the divisor $X=\varphi_*(\tilde X+E)$ on $\XX$ is Cartier as 
well.
\end{proof}

\begin{proof}[Proof of Theorem \textup{\ref{theorem-Q-smmoothings-6}}]
It is sufficient to embed $X$ to a canonical threefold singularity $(\XX\ni o)$ 
as a Cartier divisor. Let $(X^\sharp \ni o^\sharp)\to (X\ni o)$ be the index one 
cover. Then $(X^\sharp \ni o^\sharp)$ is a simple elliptic singularity (see 
\ref{index-one-cover}). In the notation of Examples~\ref{examples-index} 
consider the following $\mumu_I$-invariant elliptic curve $C\subset S$:
\par\medskip\noindent
\scalebox{1}{
\setlength{\extrarowheight}{6pt}
\begin{tabularx}{\textwidth}{@{}ll@{}}
\small{\ref{smoothing-index=6}\ref{smoothing-index=3-deg=6}}& 
$\zeta_3(u_0w_1-u_1w_0)(v_0+v_1)+ (u_0v_1-u_1v_0)(w_0+w_1)$
\\
\small{\ref{smoothing-index=4}}& 
$(u_1^2-u_0^2)v_0v_1+\zeta_4u_0u_1(v_1^2-v_0^2)$
\\
\small{\ref{smoothing-index=3-deg=9}}& 
$u_0^2u_1+ u_1^2u_2+ u_2^2 u_0$
\\
\small{\ref{smoothing-index=2-deg=6}}&
$c_1(u_0w_1-u_1w_0)( v_0 +v_1)+ c_2(u_0 v_1-u_1 v_0) (w_0+ w_1)$
\\
\small{\ref{smoothing-index=2-deg=8}}&
$c_1( u_0^2 v_0^2{-}u_1^2 v_1^2)+c_2v_0 v_1( u_0^2 {-}u_1^2 )+ c_3 (u_0^2 v_1^2{-}u_1^2 v_0^2)+c_5u_0 u_1 (v_0^2{-} v_1^2)$ 
\end{tabularx}
}

\par\medskip\noindent
where $c_i$'s are constants and $\zeta_n$ is a primitive $n$-th root of unity.
Then we apply Lemma~\ref{lemma-index}.
\end{proof}

\section{Noether's formula}
\label{sect-Noether}
\begin{proposition}[{\cite{Hacking-Prokhorov-2010}}]
Let $X$ be a projective rational surface with only rational singularities. 
Assume that every singularity of $X$ admits a $\QQ$-Gorenstein smoothing. Then
\begin{equation}
\label{equation-Noether-formula}
K_X^2+\uprho(X)+\sum_{P\in X}\upmu_P=10.
\end{equation}
\end{proposition}

\begin{proof}
Let $\eta:Y\to X$ be the minimal resolution.
Since $X$ has only rational singularities, we have
\begin{equation*}
\Eu(Y)=\Eu(X)+\sum _P\varsigma_P,\qquad 
\chi(\OOO_{Y})=\chi(\OOO_X).
\end{equation*}
Further, we can write 
\begin{equation*}
K_{Y}=\eta^* K_X-\sum_P\Delta_P,\qquad 
K_{Y}^2=K_X^2+\sum_P\Delta_P^2. 
\end{equation*}
By the usual Noether formula for smooth surfaces
\begin{equation*}
\label{equation-Noether-formula-general-proof}
12\chi(\OOO_X)=K_{Y}^2+\Eu(Y)=
K_X^2+\Eu(X)+\sum_P (\Delta_P^2+\varsigma_P).
\end{equation*}
Now the assertion follows from \eqref{equation-computation-muP}.
\end{proof}

\begin{case}
Let $X$ be an arbitrary normal projective surface, 
let $\eta:Y\to X$ be the minimal resolution, and let $D$ be a Weil divisor on $X$.
Write $\eta^*D=D_Y+D^\bullet$, where $D_Y$ is the proper transform of $D$ 
and $D^{\bullet}$ is the exceptional part of $\eta^*D$.
Define the following number
\begin{equation}
\label{equation-def-cP}
c_X(D)=-\textstyle{\frac12}\langle D^{\bullet}\rangle\cdot (\lfloor\eta^*D\rfloor-K_{Y }).
\end{equation}
\end{case}

\begin{sproposition}[{\cite[\S 1]{Blache1995}}]
\label{proposition-RR}
In the above notation we have
\begin{equation}
\label{equation-proposition-RR}
\chi(\OOO_X(D))=\textstyle{\frac12} D\cdot (D-K_X)+\chi(\OOO_X)+c_X(D)+c'_X(D),
\end{equation}
where 
\begin{equation*}
c'_X(D):=
h^0(R^1\eta_*\OOO_{Y}(\lfloor\eta^*D\rfloor))
-h^0(R^1\eta_*\OOO_{Y}).
\end{equation*}
\end{sproposition}

\begin{sremark}
Note that $c_X(D)$ can be computed locally:
\begin{equation*}
c_{X}(D)=\sum_{P\in X} c_{P,X}(D),
\end{equation*}
where $c_{P,X}(D)$ is defined by the formula \eqref{equation-def-cP}
for each germ $(X\ni P)$.
\end{sremark}

\begin{slemma}
Let $(X\ni P)$ be a rational log canonical surface singularity.
Then
\begin{equation*}
c_{P,X}(-K_X)=\Delta^2-\lceil\Delta\rceil^2-3.
\end{equation*}
where, as usual, $\Delta$ is defined by $K_Y=\eta^*K_X-\Delta$.
\end{slemma}
\begin{proof} 
Put $D:=-K_X$ and write
\begin{equation*}
\eta^*D=-K_{Y}-\Delta,\qquad\langle D^{\bullet}\rangle=\langle-\Delta\rangle=\lceil\Delta\rceil-\Delta,
\end{equation*}
\begin{equation*}
\lfloor\eta^*D\rfloor-K_{Y }=
-2K_{Y}-\lceil\Delta\rceil=-2\eta^*K_X+2\Delta-\lceil\Delta\rceil.
\end{equation*}
Therefore,
\begin{equation*}
c_{P,X}(D)=\textstyle{\frac12}(\Delta-\lceil\Delta\rceil)\cdot (-2\eta^*K_X+2\Delta-\lceil\Delta\rceil)
=\textstyle{\frac12}(\lceil\Delta\rceil-\Delta)\cdot (\lceil\Delta\rceil-2\Delta).
\end{equation*}
Since $(X\ni P)$ be a rational singularity, we have
\begin{equation*}
-2=2p_a(\lceil\Delta\rceil)-2=(\lceil\Delta\rceil-\Delta)\cdot\lceil\Delta\rceil,\quad
\lceil\Delta\rceil^2+2=\Delta\cdot\lceil\Delta\rceil
\end{equation*}
and the equality follows.
\end{proof}

\begin{scorollary}
Let $(X\ni P)$ be a rational log canonical surface singularity
such that $\K^2$ is integral.
Then
\begin{equation}
\label{equation-cP}
c_{P,X}(-K_X)=
\begin{cases}
-1&\text{in the case $(\mathrm{DV})$,}
\\
\phantom{-}0&\text{if $(X\ni P)$ is log terminal} \\& \text{or in the case $(\mathrm{nDV})$.}
\end{cases}
\end{equation}
\end{scorollary}

\begin{proof}
Let us consider the $(\mathrm{nDV})$ case (other cases are similar). By 
Proposition~\ref{Proposition-computation-K2} we have $-\Delta^2=n-9+\sum r_i$.
On the other hand, $\lceil\Delta\rceil^2=-n + 6 -\sum r_i$. Hence, 
$c_{P,X}(-K_X)=0$ as claimed.
\end{proof}

\begin{scorollary}
\label{lemma-RR}
Let $X$ be a del Pezzo surface with log canonical rational singularities and
$\uprho(X)=1$. Assume that 
for any singularity of $X$ the invariant $\K^2$
is integral. Then 
$H^i(X,\OOO_X)=0$ for $i>0$ and $\dim |-K_X|\ge K_X^2-1$.
\end{scorollary}
\begin{proof}
By the Serre duality $H^2(X,\OOO_X)=H^0(X,K_X)=0$. 
If the singularities of $X$ are rational, then the Albanese map 
is a well defined morphism $\operatorname{alb}:X\to\operatorname{Alb}(X)$.
Since $\uprho(X)=1$, we have $\dim\operatorname{Alb}(X)=0$ and so $H^1(X,\OOO_X)=0$.
The last inequality follows from \eqref{equation-proposition-RR} because 
$c'_X(-K_X)\ge 0$ and $c_X(-K_X)\ge -1$ (see \eqref{equation-cP}). 
\end{proof}

\section{Del Pezzo surfaces}
\label{section-Del-Pezzo-surfaces}
\begin{assumption}
\label{Assumptions}
From now on let $X$ be a del Pezzo surface satisfying the 
following conditions:
\begin{enumerate}
\item
the singularities of $X$ are log canonical and
$X$ has at least one non-log terminal point $o\in X$, 
\item
$X$ admits a $\QQ$-Gorenstein smoothing,
\item
$\uprho(X)=1$.
\end{enumerate}
\end{assumption}

\begin{lemma}
\label{lemma-RR-and-one-point}
In the above assumptions the following hold:
\begin{enumerate}
\item \label{lemma-RR-and-one-point1}
$\dim |-K_X|>0$,
\item \label{lemma-RR-and-one-point2}
$X$ has exactly one non-log terminal point.
\end{enumerate}
\end{lemma}

\begin{proof}
\ref{lemma-RR-and-one-point1} is implied by semicontinuity (cf. \cite[Theorem 4]{Manetti-1991}).
\ref{lemma-RR-and-one-point2}
follows from Shokurov's connectedness theorem
\cite[Lemma 5.7]{Shokurov-1992-e-ba}, \cite[Th. 17.4]{Utah}.
\end{proof}

\begin{construction}
\label{construction-main}
Let $\sigma:\tilde X\to X$ be a dlt modification and let 
\begin{equation*}
\tilde C=\sum_{i=1}^{s}\tilde C_i=\sigma^{-1}(o)
\end{equation*}
be the exceptional divisor. 
Thus $\uprho(\tilde X)=s+1$.

For some large $k$ the divisor $-kK_X$ is very ample.
Let $H\in |-kK_X|$ be a general member and let $\Theta:=\frac 1k H$.
Then $K_X+\Theta\equiv 0$ and the pair $(X,\Theta)$ is lc at $o$ and klt outside $o$. 
We can write 
\begin{equation}
\label{equation-crepant-formula}
K_{\tilde X}+\tilde C=\sigma^*K_X,
\qquad 
K_{\tilde X}+\tilde\Theta+\tilde C=\sigma^*(K_X+\Theta),
\end{equation}
where $\tilde \Theta$ is the proper transform of $\Theta$ on $\tilde X$. 
Clearly $\tilde C\cap\Supp(\tilde\Theta)=\emptyset$ and $\tilde \Theta$ is nef and big.
Note also that $K_{\tilde X}$ is $\sigma$-nef.
\end{construction}

\begin{scase}
\label{case-I=6-tildeX}
Let $D\in |-K_X|$ be a member such that $o\in\Supp(D)$.
This holds automatically for any member $D\in |-K_X|$
if $I>1$ because $-K_X$ is not Cartier at 
$o$ in this case.
In general, such a member exists by Lemma~\ref{lemma-RR-and-one-point}\ref{lemma-RR-and-one-point1}.
We have 
\begin{equation}
\label{equation-I=6-tildeX}
K_{\tilde X}+\sum m_i\tilde C_i+\tilde D\sim 0,\quad m_i\ge 2\quad \forall i.
\end{equation}
\end{scase}

\begin{case}
We distinguish two cases that will be treated in Sect.~\ref{section-fibrations}
and~\ref{section-del-pezzo} respectively:
\begin{enumerate}
\renewcommand\labelenumi{{\rm (\Alph{enumi})}}
\renewcommand\theenumi{{\rm (\Alph{enumi})}}
\item 
\label{division-fibrations}
there exists a fibration $\tilde X\to T$ over a smooth curve,
\item 
\label{division-del-pezzo}
$\tilde X$ has no dominant morphism to a curve.
\end{enumerate}
\end{case}

Note that the divisor  $-(K_{\tilde X}+\tilde C)$ is nef and big.
Therefore, in the case  \ref{division-fibrations} the generic fiber
of the fibration $\tilde X\to T$ is a smooth rational curve.

To show the existence of $\QQ$-Gorenstein smoothings we use unobstructedness 
of deformations:

\begin{proposition}[{\cite[Proposition 3.1]{Hacking-Prokhorov-2010}}]
\label{no-obstructions}
Let $Y$ be a projective surface with log canonical singularities 
such that $-K_Y$ is big. Then there are no local-to-global 
obstructions to deformations of $Y$.
In particular, if the singularities of $Y$ admit $\QQ$-Gorenstein smoothings, then 
the surface $Y$ admits a $\QQ$-Gorenstein smoothing. 
\end{proposition}

However, in some cases the corresponding smoothings can be constructed explicitly:
\begin{sexample}
Consider the hypersurface $X\subset\PP(1,1,2,3)$ given by $z^2=y\phi_4(x_1,x_2)$.
Then $X$ is a del Pezzo surface with $K_X^2=1$.
The singular locus of $X$ consists of 
the point $(0:0:1:0)$ of type $[3; [2]^4]$
and four points
$\{z=y=\phi_4(x_1,x_2)=0\}$ of types $\rA_1$. 
Therefore, $X$ is of type~\ref{types-I=2-4-I2} with $n=3$.
\end{sexample}

\begin{sexample}
Consider the hypersurface $X\subset\PP(1,1,2,3)$ given by $(x_1^3-x_2^3)z+y^3=0$.
Then $X$ is a del Pezzo surface with $K_X^2=1$. The singular locus of $X$ consists of 
the point $(0:0:0:1)$ of type $[2; [3]^3]$
and three points
$(1:\zeta_3^k:0:0)$, $k=0,1,2$ of type $\rA_2$. Therefore, 
$X$ is of type~\ref{types-I=3}
with $n=2$.
\end{sexample}

\section{Proof of Theorem~\ref{theorem-main}: Fibrations}
\label{section-fibrations}
In this section we consider the case~\ref{division-fibrations} of~\ref{construction-main}.
First we describe quickly the singular fibers that occur in our classification.

\begin{case}
\label{List-A}
Let $Y$ be a smooth surface and let $Y\to T$ be a rational curve fibration.
Let $\Sigma\subset Y$ be a section and let $F$ be a singular fiber.
We say that $F$ is of type 
\type{(I_k)} or \type{(II)}
if its dual graph has the following form,
where $\square$ corresponds to $\Sigma$ and $\bullet$ corresponds to a $(-1)$-curve:
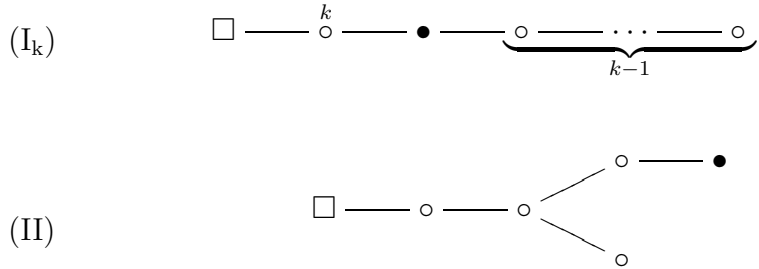
\begin{equation*}
\vcenter{
\xy
\xymatrix"M"{
\square\ar@{-}[r]&
\overset{k}\circ\ar@{-}[r]&\bullet\ar@{-}[r]&\circ\ar@{-}[r]&\cdots\ar@{-}[r]&\circ&
} 
\POS"M1,4"."M1,6"!C*\frm{_\}},-U*---!D\txt{$\scriptstyle{k-1}$}
\endxy
\vspace{16pt}
}
\leqno{\mathrm{(I_k)}}
\end{equation*}

\begin{equation*}
\vcenter{
\xy
\xymatrix@R=5pt{
&&&\circ\ar@{-}[r]&\bullet
\\
\square\ar@{-}[r]&\circ\ar@{-}[r]&\circ\ar@{-}[rd]\ar@{-}[ru]
\\
&&&\circ
}
\endxy 
\vspace{16pt}
}
\leqno{\mathrm{(II)}}
\end{equation*}
Assume that $Y$ has only fibers of these types 
\type{(I_k)} or \type{(II)}.
Let $Y\to \bar X$ be the contraction of all curves in fibers 
having self-intersections less than $-1$, i.e. corresponding to white vertices.
Then $\uprho(\bar X)=2$ and $\bar X$ has a contraction $\theta: \bar X\to T$.

\begin{sremark}
Let $\bar C\subset \bar X$ be the image of $\Sigma$.
Assume that $\bar X$ is projective,  $\bar C^2<0$, i.e. $\bar C$ is contractible,
and $(K_{\bar X}+\bar C)\cdot\bar C=0$.
For a general fiber $F$ of $\theta$ we have $(K_{\bar X}+\bar C)\cdot F=-1$.
Therefore, $-(K_{\bar X}+\bar C)$ is nef.
Now let $\bar X\to X$ be the contraction of $\bar C$. Then $X$ is a del Pezzo surface with 
$\uprho(X)=1$.
\end{sremark}
\end{case}

\begin{case}
Recall that we use the notation of~\ref{Assumptions} and~\ref{construction-main}. In this section 
we assume that $\tilde X$ has a rational curve fibration $\tilde X\to T$, where $T$ is a smooth curve
(the case~\ref{division-fibrations}).
Since $\uprho(X)=1$, the curve $\tilde C$ is not contained in the fibers.
A general fiber $\tilde F\subset\tilde X$ is a smooth rational curve.
By the adjunction formula $K_{\tilde X}\cdot\tilde F=-2$. By \eqref{equation-I=6-tildeX}
we have $\tilde F\cdot\sum m_i\tilde C_i=2$ and so $\tilde F\cdot\tilde D=0$.
Hence there exists exactly one component of $\tilde C$, say $\tilde C_1$, such that 
$\tilde F\cdot\tilde C_1=1$, $m_1=2$, and for $i\neq 1$ we have 
$\tilde F\cdot\tilde C_i=0$. This means that 
the divisor $\tilde D$ and the components $\tilde C_i$ with $i\neq 1$ are contained in the fibers
and $\tilde C_1$ is a section of the fibration $\tilde X\to T$.

Let us contract all the vertical components of $\tilde C$, i.e. the components $\tilde C_i$
with $i\neq 1$.
We get the following diagram 
\begin{equation*}
\xymatrix@C=80pt{
\tilde X\ar[d]_{\sigma}\ar[r]^{\nu}
&\bar X\ar[d]^{\theta}\ar[dl]
\\
X&T
}
\end{equation*}
Let $\bar C:=\nu_*\tilde C=\nu_*\tilde C_1$, $\bar\Theta=\nu_*\tilde\Theta$,
and $\bar D=\nu_*\tilde D$.
By \eqref{equation-crepant-formula} and \eqref{equation-I=6-tildeX}
we have 
\begin{equation}
\label{equation-I=6-barX-fibration}
K_{\bar X}+\bar C+\bar\Theta\equiv 0,\qquad
K_{\bar X}+2\bar C+\bar D\sim 0.
\end{equation}
Moreover, the pair $(\bar X,\bar C+\bar\Theta)$ is lc and if $I>1$, then $\dim |\bar D|>0$.
\end{case}

\begin{lemma}[cf. {\cite{Fujisawa1995}}]
\label{lemma-elliptic}
If the singularity $(X\ni o)$ is not rational, then 
$T$ is an elliptic curve, $\tilde X\simeq \bar X$ is smooth,
and $X$ is a generalized cone over $T$.
\end{lemma}

\begin{proof}
By Theorem~\ref{theorem-classification-lc-singularities}\ref{Theorem-simple-elliptic-cusp=I=1} the surface $\tilde X$ is smooth along $\tilde C$.
Since $\tilde C_1$ is a section, we have $\tilde C_1\simeq T$ and $\tilde C$ cannot be a combinatorial 
cycle of smooth rational curves.
Hence both $\tilde C_1$ and $T$ are smooth elliptic curves.
Then $\tilde C=\tilde C_1$ and
$\uprho(\tilde X)=\uprho(X)+1=2$.
Hence any fiber $\tilde F$ of the fibration $\tilde X\to T$ is irreducible.
Since $\tilde F\cdot\tilde C_1=1$, any fiber is not multiple.
This means that $\tilde X\to T$ is a smooth morphism.
Therefore, $\tilde X$ is a geometrically ruled surface over an elliptic curve.
\end{proof}

From now on we assume that the singularities of $X$ are rational.
In this case, $T\simeq\PP^1$ and $\dim |\bar D|\ge \dim |-K_X|>0$ (see~\ref{case-I=6-tildeX}
and Lemma~\ref{lemma-RR-and-one-point}).

\begin{lemma}
Let $\bar F$ be a degenerate fiber \textup(with reduced structure\textup).
Then the dual graph of $\bar F$ has one of the forms described in \xref{List-A}:
\par\smallskip
\type{(I_k)} with $k=2$, $3$, $4$ or $6$, or \type{(II)}.
\end{lemma}

\begin{proof}
Let $\bar P:=\bar C\cap\bar F$.
Since $-(K_{\bar X}+\bar C+\bar F)$ is $\theta$-ample,
the pair $({\bar X},\bar C+\bar F)$ is plt outside $\bar C$ by Shokurov's connectedness theorem.
Let $m$ be the multiplicity of $\bar F$.
Since $\bar C$ is a section of $\theta$, we have $\bar C\cdot\bar F=1/m<1$ and so
the point $\bar P\in\bar X$ is singular. 

If the pair $(\bar X,\bar F)$ is plt at $\bar P$, then $\bar X$ has 
on $\bar F$ two singular points and these points are of types 
$\frac 1n (1,q)$ and $\frac 1n (1,-q)$ (see e.g. \cite[Th. 7.1.12]{Prokhorov-2001}). 
We may assume that $\bar P\in\bar X$ is of type $\frac 1n (1,q)$.
In this case, $m=n$ and the pair $({\bar X},\bar C+\bar F)$ is lc at $\bar P$
because $\bar C\cdot\bar F=1/n$. 
By Theorem~\ref{Theorem-Q-smoothings}
we have $n=2$, $3$, $4$, or $6$ and $q=1$.
We get the case \type{(I_k)}.
From now on we assume that $(\bar X,\bar F)$ is not plt at $\bar P$.
In particular, $(\bar X\ni\bar P)$ is not of type $\frac 1n(1,1)$.
Then again by Theorem~\ref{Theorem-Q-smoothings} the singularity $(o\in X)$
is of type $[n_1,\dots, n_s; [2]^4]$. Hence the part of the dual graph of $F$ 
attached to $C_1$ has the form
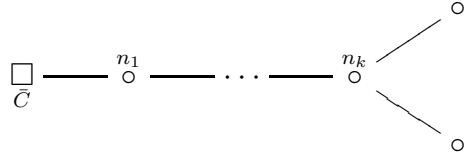
\begin{equation*}
\vcenter{
\xy
\xymatrix@R=5pt{ 
&&&&\circ
\\
\underset{\bar C}\square\ar@{-}[r]&\overset{n_1}\circ\ar@{-}[r]&\cdots\ar@{-}[r]&\overset{n_k}\circ\ar@{-}[ru]\ar@{-}[rd]&
\\
&&&&\circ
}
\endxy 
}
\end{equation*}
where $k\ge 1$.
Then $K_{\bar X}+\bar C$ is of index $2$ at $\bar P$ (see \cite[Prop. 16.6]{Utah}).
Since $(K_{\bar X}+\bar C)\cdot m\bar F=-1$, the number
$2(K_{\bar X}+\bar C)\cdot \bar F=-2/m$
must be an integer.
Therefore, $m=2$. Assume that $\bar X$ has a singular point $\bar Q$
on $\bar F\setminus\{\bar P\}$. We can write $\Diff_{\bar F}(0)=\alpha_1\bar P+\alpha_2\bar Q$,
where $\alpha_1\ge 1$ (by the inversion of adjunction) and $\alpha_2\ge 1/2$. Then $\Diff_{\bar F}(\bar C)=\alpha_1'\bar P+\alpha_2\bar Q$,
where $\alpha_1'=\alpha_1+\bar F\cdot\bar C\ge 3/2$. On the other hand, the divisor
\[
-(K_{\bar F}+\Diff_{\bar F}(\bar C))=-(K_{\bar X}+\bar F+\bar C)|_{\bar F}
\]
is ample. Hence, $\deg\Diff_{\bar F}(\bar C)<2$, a contradiction.
Thus $\bar P$ is the only singular point of $\bar X$ on $\bar F$.
We claim that $\bullet$ is attached to 
one of the $(-2)$-curves at the end of the graph.
Indeed, assume that 
the dual graph of $F$ 
has the form
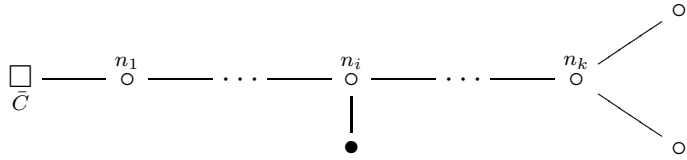
\begin{equation*}
\vcenter{
\xy
\xymatrix@R=5pt{ 
&&&&&&\circ
\\
\underset{\bar C}\square\ar@{-}[r]&\overset{n_1}\circ\ar@{-}[r]&\cdots\ar@{-}[r]&\overset{n_i}\circ\ar@{-}[r]&\cdots\ar@{-}[r]&\overset{n_k}\circ\ar@{-}[ru]\ar@{-}[rd]&
\\
&&&\bullet\ar@{-}[u]&&&\circ
}
\endxy 
}
\end{equation*}
where $1\le i\le k$. Clearly, $n_i=2$.
Contracting the $(-1)$-curve $\bullet$ we obtain the following graph
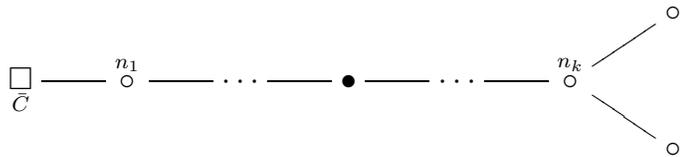
\begin{equation*}
\vcenter{
\xy
\xymatrix@R=5pt{ 
&&&&&&\circ
\\
\underset{\bar C}\square\ar@{-}[r]&\overset{n_1}\circ\ar@{-}[r]&\cdots\ar@{-}[r]&\bullet\ar@{-}[r]&\cdots\ar@{-}[r]&\overset{n_k}\circ\ar@{-}[ru]\ar@{-}[rd]&
\\
&&&&&&\circ
}
\endxy 
}
\end{equation*}
Continuing the process, on each step we have a configuration of the same type 
and finally we get the dual graph 
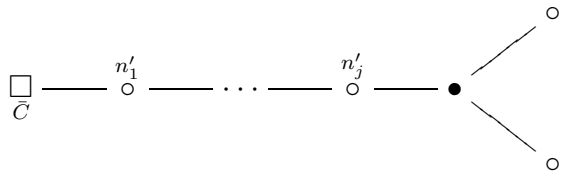
\begin{equation*}
\vcenter{
\xy
\xymatrix@R=5pt{ 
&&&&&\circ
\\
\underset{\bar C}\square\ar@{-}[r]&
\overset{n_1'}\circ\ar@{-}[r]&\cdots\ar@{-}[r]&\overset{n_{j}'}\circ\ar@{-}[r]&
\bullet\ar@{-}[ru]\ar@{-}[rd]&
\\
&&&&&\circ
}
\endxy 
}
\end{equation*}
where $j\ge 0$. Then the next contraction gives us a configuration which is not a 
simple normal crossing divisor. The contradiction proves our claim.
Similar arguments show that  
$n_k=n_{k-1}=2$ and $k=2$,
i.e. we get the case \type{(II)}.
\end{proof}

\begin{proof}[Proof of Theorem \xref{theorem-main} in the case \xref{construction-main}\xref{division-fibrations}]
If all the fibers are smooth, then by Lemma~\ref{lemma-elliptic} we have the case~\ref{types-simple-elliptic}.
If there exist a fiber of type \type{(I_k)} with $k>2$, then
$I>2$ and by Theorem~\ref{Theorem-Q-smoothings} we have cases~\ref{types-I=3},~\ref{types-I=4-I22I4},~\ref{types-I=6}.
If all the fibers are of types \type{(I_2)} or \type{(II)}, then $I=2$ 
and we have cases~\ref{types-I=2-4-I2},~\ref{types-I=2-2I2-IV},~\ref{types-I=2-2IV}.
The computation of $K_X^2$ follows from \eqref{equation-Noether-formula}
and \eqref{equation-computation-muP-nG}.
\end{proof}

\section{Proof of Theorem~\ref{theorem-main}: Birational contractions}
\label{section-del-pezzo}
\begin{case}
In this section we assume that $\tilde X$ has no dominant morphism to a curve
(case~\ref{construction-main}\ref{division-del-pezzo}). It will we shown that this case does not occur.

Run the $K_{\tilde X}$-MMP on $\tilde X$.
Since $-K_{\tilde X}$ is big, on the last step we get a Mori fiber space $\bar X\to T$
and by our assumption $T$
cannot be a curve. Hence $T$ is a point and $\bar X$ is a del Pezzo surface with 
$\uprho(\bar X)=1$.
Moreover, the singularities of $\bar X$ are log terminal and so $\bar X\not\simeq X$.
Thus we get the following diagram 
\begin{equation*}
\xymatrix@R=0.7pc{
&\tilde X\ar[dl]_{\sigma}\ar[rd]^{\nu} 
\\
X\ar@{-->}[rr]&&\bar X
}
\end{equation*}
Put $\bar C:=\nu_* \tilde C$ and $\bar C_i:=\nu_* \tilde C_i$.
By \eqref{equation-I=6-tildeX} we have
\begin{equation}
\label{equation-I=6-barX}
K_{\bar X}+\sum m_i\bar C_i+\bar D\sim 0,\qquad m_i\ge 2.
\end{equation}
Since $\uprho(X)=\uprho(\bar X)$ and $\tilde C$ is the 
$\sigma$-exceptional divisor, the whole $\tilde C$ cannot be 
contracted by $\nu$.
\end{case}

\begin{lemma}
\label{lemma-fiber-meets-C}
Any fiber $\nu^{-1}(\bar P)$ of positive dimension meets $\tilde C$.
\end{lemma}

\begin{proof}
Since $\bar X$ is normal, $\nu^{-1}(\bar P)$ is a connected contractible effective divisor.
Since all the components of $\tilde C$ are $K_{\tilde X}$-non-negative,
$\nu^{-1}(\bar P)\not\subset\tilde C$.
Since $\uprho(X)=1$, we have $\nu^{-1}(\bar P)\cap\tilde C\neq \emptyset$.
\end{proof}

\begin{lemma}
\label{lemma-fiber-meets-C-1}
If $\nu$ is not an isomorphism over $\bar P$, then $(\bar X,\bar C)$ is plt at $\bar P$.
In particular, $\bar C$ is smooth at $\bar P$.
\end{lemma}

\begin{proof}
Since $K_{\tilde X}+\tilde C+\tilde\Theta\equiv 0$,
the pair $(\bar X,\bar C+\bar\Theta)$ is lc.
By the above lemma there exists a component $\tilde E$ of $\nu^{-1}(\bar P)$
meeting $\tilde C$.
By Kodaira's lemma the divisor $\tilde\Theta-\sum\alpha_i\tilde C_i$ is ample for
some $\alpha_i>0$.
Hence $\tilde E$ meets $\tilde\Theta$ and so $\Supp (\bar\Theta)$ contains $\bar P$.
Therefore, $(\bar X,\bar C)$ is plt at $\bar P$.
\end{proof}

\begin{scorollary}
\label{corollary-dlt}
$(\bar X,\bar C)$ is dlt. 
\end{scorollary}

\begin{lemma}
\label{lemma-barC-irrducible}
\begin{enumerate}
\item
\label{lemma-barC-irrducible-1}
$\bar C$ is an irreducible smooth rational curve; 
\item
\label{lemma-barC-irrducible-2}
$\bar X$ has at most two singular points on $\bar C$;

\item 
\label{lemma-barC-irrducible-3}
the singularities of $X$ are rational \textup{(see also \cite[Corollary 1.9]{Fujisawa1995})}.
\end{enumerate}
\end{lemma}
\begin{proof}
\ref{lemma-barC-irrducible-1}
Let $\bar C_1\subset\bar C$ be any component meeting $\bar D$
and let $\bar C':=\bar C-\bar C_1$.
Assume that $\bar C'\neq 0$. By~\ref{corollary-dlt} any point 
$\bar P\in\bar C_1\cap\bar C'$ is a smooth point of 
$\bar X$. 
Hence $\Diff_{\bar C_1}(\bar C')$ contains $\bar P$ with positive integral 
coefficient and $\deg\Diff_{\bar C_1}(\bar D+\bar C')\ge 2$ because 
$\Supp (\bar D)\cap\bar C\neq\emptyset$.
On the other hand, $-(K_{\bar X}+\bar C+\bar D)$ is ample
by \eqref{equation-I=6-barX}. Thus contradicts the adjunction formula. Thus $\bar C$ is irreducible.
Again by the adjunction
\begin{equation*}
\deg K_{\bar C}+\deg\Diff_{\bar C}(0)<0.
\end{equation*}
Hence, $p_a(\bar C)=0$.

\ref{lemma-barC-irrducible-2}
Assume that $\bar X$ is singular at $\bar P_1,\dots,\bar P_N\in\bar C$. 
Write 
\begin{equation*}
\Diff_{\bar C}(0)=\sum_{i=1}^N\left(1-\textstyle{\frac 1{b_i}}\right)\bar P_i
\end{equation*}
for some $b_i\ge 2$.
The coefficient of $\Diff_{\bar C}(\bar D)$ at points of the intersection
$\Supp(\bar D)\cap\bar C$ is at least $1$.
Since $\Supp (\bar D)\cap\bar C\neq\emptyset$, we have $N\le 2$.

\ref{lemma-barC-irrducible-3}
If $(X\ni o)$ is a non-rational singularity, then $p_a(\tilde C)=1$ and $\tilde X$ is smooth along $\tilde C$.
Hence $p_a(\bar C)\ge1$. This contradicts~\ref{lemma-barC-irrducible-1}.
\end{proof}

\begin{lemma}
\label{lemma-contractions-K1}
Let $\varphi:S\to S'$ be a birational Mori contraction of surfaces with log terminal singularities
and let $E\subset S$ be the exceptional divisor.
Then $-K_S\cdot E\le 1$ and the equality holds if and only if the singularities of $S$ 
along $E$ are at worst Du Val.
\end{lemma}
\begin{proof} 
Let $\psi:S^{\min}\to S$ be the minimal resolution and let $\tilde E\subset S^{\min}$ 
be the proper transform of $E$.
Write $K_{S^{\min}}=\psi^*K_S-\Delta$. 
Since $K_{S^{\min}}\cdot\psi^* E<0$, the divisor $K_{S^{\min}}$ is not nef over $Z$.
Hence, $K_{S^{\min}}\cdot\tilde E=-1$ and so
$-K_S\cdot E+\tilde E\cdot\Delta=1$. 
\end{proof}

\begin{lemma}
\label{lemma-E}
Let $\nu':\tilde X\to X'$ be the first extremal contraction in $\nu$
and let $\tilde E$ be its exceptional divisor.
Then $\tilde E\not\subset\tilde C$. Moreover,
$\tilde E\cap\tilde C$ is a singular point of $\tilde X$
and smooth point of $\tilde C$.
\end{lemma}
\begin{proof}
Since $\uprho(X)=1$, \ $\tilde E\cap \tilde C\neq\emptyset$.
Since $K_{\tilde X}$ is $\sigma $-nef, $\tilde E\not\subset\tilde C$.
Since $\bar C$ is a smooth rational curve, $\tilde E$ meets $\tilde C$ at a single point,
say $\tilde P$. Further, $\sigma (\tilde E)$ meets $\Supp(\Theta)$ outside $o$. Hence, $\tilde\Theta\cdot\tilde E>0$.
By Lemma~\ref{lemma-contractions-K1}\ $K_{\tilde X}\cdot\tilde E\ge -1$. Since
$K_{\tilde X}+\tilde C+\tilde\Theta\equiv 0$,
we have $\tilde C\cdot\tilde E<1$.
Hence $\tilde C\cap\tilde E$ is a singular point of $\tilde X$.
Since $(\tilde X,\tilde C)$ is dlt, $\tilde C\cap\tilde E$ is a smooth point of $\tilde C$
(see e.g. \cite[16.6]{Utah}).
\end{proof}

\begin{proposition}
\label{proposition-tilde-C-irreducible}
$\uprho(\tilde X)=2$ and $\tilde C$ is irreducible.
Moreover, $\bar X$ has exactly two singular points on $\bar C$ and $I>2$.
\end{proposition}
\begin{proof}
Assume the converse, i.e. $\tilde C$ is reducible.
By Lemma~\ref{lemma-barC-irrducible} the curve $\bar C$
is irreducible. Let $s$ be the number of components of $\tilde C$. 
So, $\uprho(\tilde X)=s+1$. Hence $\nu$ contracts $s-1$ components of $\tilde C$ 
and exactly one divisor, say $\tilde E$ such that $\tilde E\not\subset\tilde C$.
By Lemma~\ref{lemma-E} the curve $\tilde E$ is contracted on the first step. Note that $\tilde C$ is a chain 
$\tilde C_1+\cdots+\tilde C_s$, where both $\tilde C_1$ and $\tilde C_s$ contain
two points of type $\rA_1$ and the middle curves $\tilde C_2$,\dots, $\tilde C_{s-1}$ 
are contained in the smooth locus. By Lemma~\ref{lemma-E}
we may assume that $\tilde E$ meets $\tilde C_1$.
Then $\nu$ contracts $\tilde C_1$,\dots, $\tilde C_{s-1}$.
However $\tilde C_s$ contains two points of type $\rA_1$ 
and it is not contracted. Thus $\bar X$ has two singular points of type $\rA_1$ 
on $\bar C$. Again by Lemma~\ref{lemma-barC-irrducible} the surface
$\bar X$ has no other singular points on $\bar C$.
In particular, $2\bar C$ is Cartier, $\bar X$ has only singularities of type $\rT$,
and $K_{\bar X}^2$ is an integer.
On the other hand, we have $-K_{\bar X}=m\bar C+\bar D$, $m\ge 2$. 
By the adjunction formula
\begin{equation*}
-1=\deg(K_{\bar C}+\Diff_{\bar C}(0))=(K_{\bar X}+\bar C)\cdot\bar C=
-\bar D\cdot\bar C-(m-1)\bar C^2.
\end{equation*}
This gives us $\bar D\cdot\bar C=\bar C^2=1/2$, $m=2$, 
and $K_{\bar X}^2=9/2$, a contradiction.

Finally, by Lemmas~\ref{lemma-barC-irrducible} and~\ref{lemma-E} the surface 
$\tilde X$ (resp. $\bar X$) has exactly three (resp. two) singular points on $\tilde C$.
\end{proof}

By Theorem~\ref{Theorem-Q-smoothings} the surface $\bar X$ has at least one non-Du Val singularity
lying on $\bar C$. Thus 
Theorem~\ref{theorem-main} is implied by the following.
\begin{proposition}
\label{lemma-DuVal-singularities-C}
$\bar X$ has only Du Val singularities on $\bar C$. 
\end{proposition}
\begin{proof}
Assume that the singularities of $\bar X$ 
at points lying on $\bar C$ are of types $\frac 1{n_1} (1,1)$
and $\frac 1{n_2} (1,1)$ with $n_1\ge n_2$ and $n_1>2$.
In this case near $\bar C$ the divisor $H:=-(K_{\bar X}+2\bar C)$ is Cartier.
By the adjunction formula
\begin{equation*}
K_{\bar C}+\Diff_{\bar C}(0)=(K_{\bar X}+\bar C)|_{\bar C}
=-(H+\bar C)|_{\bar C}
\end{equation*}
Hence, 
\begin{equation*}
\deg\Diff_{\bar C}(0)< 2-H\cdot\bar C\le 1.
\end{equation*}
In particular, $\bar X$ has at most one singular point on $\bar C$,
a contradiction. 
\end{proof}


\def\cprime{$'$}

\end{document}